\documentclass[10pt,amstex]{article}
\usepackage{amsmath}
\usepackage{amsfonts}
\usepackage{amssymb}
\usepackage{amsthm}
\usepackage{graphicx}
\begin{document}

\newtheorem{lem}{Lemma}
\newtheorem{defn}{Definition}
\newtheorem{prop}{Proposition}
\newtheorem{thm}{Theorem}
\newtheorem{cor}{Corollary}

\title{Classification of Minimal Separating Sets in Low Genus Surfaces}
\author{J. J. P. Veerman\thanks{Fariborz Maseeh Dept. of Math. and Stat., Portland State Univ., Portland, OR, USA; e-mail: veerman@pdx.edu}, William Maxwell\thanks{Fariborz Maseeh Dept. of Math. and Stat., Portland State Univ., Portland, OR, USA}, Victor Rielly\thanks{Fariborz Maseeh Dept. of Math. and Stat., Portland State Univ., Portland, OR, USA}, Austin K. Williams\thanks{Fariborz Maseeh Dept. of Math. and Stat., Portland State Univ., Portland, OR, USA}
}
\maketitle

\begin{abstract}
Consider a surface $S$ and let $M\subset S$. If $S\setminus M$ is not connected, then we say $M$ \emph{separates} $S$, and we refer to $M$ as a \emph{separating set} of $S$. If $M$ separates $S$, and no proper subset of $M$ separates $S$, then we say $M$ is a \emph{minimal separating set} of $S$.
In this paper we use computational methods of combinatorial topology to classify the minimal
separating sets of the orientable surfaces of genus $g=2$ and $g=3$. The classification for genus
0 and 1 was done in earlier work, using methods of algebraic topology.
\end{abstract}


\section*{Surfaces and Separating Sets}
We begin with some basic definitions.

\begin{defn}
\label{def:surface}
When we say \emph{surface} we are referring to a closed, connected, triangulated, orientable, 2-manifold.
\end{defn}

Let $S$ be a surface and let $M\subset S$.

\begin{defn}
\label{def:separates}
If $S\setminus M$ is not connected, then we say $M$ \emph{separates} $S$, and we refer to $M$ as a \emph{separating set} of $S$.
\end{defn}

\begin{defn}
\label{def:minsepset}
If $M$ separates $S$, and no proper subset of $M$ separates $S$, then we say $M$ is a \emph{minimal separating set} of $S$.
\end{defn}

In this paper we will not consider pathological separating sets, such as the common boundary of the Lakes of Wada. To that end, when we refer to subsets of a surface it will be understood that we are referring to finitely triangulated subsets. Such subsets may be thought of as finite simplicial subcomplexes of the triangulated surface. Such sets occur naturally in the following
geometric context. Let $S$ be a connected compact surface whose distance $d(.,.)$ is inherited
from a Riemannian metric. Given two distinct points $p$ and $q$ in $S$, the mediatrix
is the set of points $x$ such that $d(x,q)=d(x,p)$. Brillouin zones were introduced by
Brillouin in the 30s and play an important role in physics \cite{Mermin}
and computational geometry \cite{Preparata}. Their boundaries are formed by pieces of of mediatrices\cite{Peixoto}. Mediatrices also play an important role in the study of focal decomposition
\cite{Peixoto82}.

In \cite{Bernhard/Veerman} and \cite{Veerman/Bernhard}
it is proved that mediatrices are closed simplicial 1-complexes. Vice versa, as we will argue in the
concluding remarks of this paper, every minimal separating simplicial 1-complex can be realized as
a mediatrix.
In the cited papers, the topological classification of minimally
(triangulable) sets up to graph isomorphism was given for the (oriented) surfaces with genus 0 and 1.
Namely, the 2-sphere admits only one minimal separating set (a simple, closed loop), and
the torus admits 5 topologically distinct minimal separating sets.

Here we carry out the classification up to graph isomorphism for the (oriented) genus 2
and 3 surfaces. Our proofs differ substantially from those in \cite{Bernhard/Veerman} and \cite{Veerman/Bernhard} where methods
of algebraic topology were used. In this paper we use combinatorial topology and combine it with
a substantial computational effort to complete the classification.

\section*{Graphs and Embeddings}

Colloquially speaking, when we talk about \emph{graphs} we are referring to undirected
multigraphs which may or may not have loops. This is made rigorous as follows.

\begin{defn}
\label{def:graph}
A \emph{graph} $G$ is a pair $(V(G),E(G))$ where $V(G)$ is a non-empty finite set of
elements called \emph{vertices}, and $E$ is a finite multiset whose elements are
unordered multiset pairs of vertices in $V(G)$. Elements of $E(G)$ are referred to as
\emph{edges} and the elements of edges are referred to as \emph{edge-ends}. Two graphs
$G$ and $H$ are the same ($G=H$) if there is a bijection from
$V(G)$ to $V(H)$ that preserves the edges.
\end{defn}

When the graph in question is clear from context, we may use the simpler notation $V$ and
$E$ to refer to the vertex set and edge set, respectively, of a given graph.

\begin{defn}
\label{def:incident}
A vertex contained in an edge is said to be \emph{incident} to the edge.
\end{defn}

\begin{defn}
\label{def:adjacent}
Vertices $v_1, v_2 \in V$ are said to be \emph{adjacent} if there exists an edge that contains them both.
\end{defn}

\begin{defn}
\label{def:loop}
An edge whose edge-ends are identical is called a \emph{loop}. An edge that is not a loop is called a \emph{non-loop}.
\end{defn}

\begin{defn}
\label{def:degree}
Let $v$ be a vertex of a graph. Let $l$ and $n$ be the number of loops and non-loops, respectively, incident to $v$. Then the \emph{degree of $v$}, denoted $deg(v)$, is given by $n+2l$.
\end{defn}

Informally, the degree is the number of edge-ends incident to the vertex.

\begin{defn}
\label{def:induced-subgraph}
Let $G$ be a graph and let $A\subset V$. Then the \emph{subgraph of $G$ induced by $A$},
denoted $G[A]$, is the graph whose vertex set is $A$, and whose edge set is the subset of
$E(G)$ all of whose edge-ends are in $A$.
\end{defn}

Less formally, the induced subgraph $G[A]$ is the graph obtained from $G$ by removing all vertices from $G$ except those in $A$. Edges that lose an edge-end during this process are also removed.

We often want to consider what parts of $G$ remain after the removal of some subset $B$
of its vertices. For example, let $B \subset V$. Then removal of the vertices in $B$ from the
graph $G$ will yield the induced subgraph $G[V \setminus B]$. We may think of the induced
subgraph $G[V\setminus B]$ as the part of $G$ that remains when we remove $B$. Thus, for
simplicity, we use the notation $G \setminus B$ to denote the induced subgraph
$G[V \setminus B]$.

\begin{defn}
\label{def:graph-homeomorphic}
Two graphs $G$ and $H$ are said to be \emph{graph-homeomorphic}
if there is a graph isomorphism from a subdivision of $G$ to a subdivision
of $H$.
\end{defn}

Recall that a graph $G'$ is a subdivision of a graph $G$ if we get $G'$ by subdividing
one or more edges in $G$: replace an edge $\{u,v\}$ with a pair of edges $\{u,w\}$ and $\{w,v\}$
where $w$ is a new vertex.
One can easily verify that graph-homeomorphism is an equivalence relation on the set of all graphs.

\begin{defn}
\label{def:homeomorphically-irreducible}
A graph $G$ is said to be \emph{homeomorphically
irreducible} if whenever a graph $H$ is homeomorphic to $G$, either
$G=H$ or $H$ is a subdivision of $G$.
\end{defn}
Behzad and Chartrand show in \cite{Behzad} that in every equivalence class
of homeomorphic graphs there exists a unique homeomorphically irreducible
graph.

\begin{defn}
\label{def:embedding}
An \emph{embedding} of a graph $G=(V,E)$ into a surface $S$ is a function $\phi$ that maps distinct vertices of $G$ to distinct points in $S$ and edges of $G$ to simple arcs (homomorphic images of $[0,1]$) in $S$ such that:
\begin{enumerate}
\item for every edge $e=\{ v_1 , v_2 \}$ in $E$, the endpoints of the simple arc $\phi(e)$ are the points $\phi(v_1)$ and $\phi(v_2)$,
\item for all $e \in E$, and for all $v \in V$, $\phi(v)$ is not in the interior of $\phi(e)$,
\item for all $e_1 , e_2 \in E$, $\phi(e_1)$ and $\phi(e_2)$ do not intersect at a point that is interior to either arc.
\end{enumerate}
\end{defn}

Note that if $\phi$ is an embedding of a graph $G$ into a surface $S$, then the $\phi[G]$ is a simplicial 1-subcomplex of $S$.

\begin{defn}
\label{def:regions}
The \emph{regions} of an embedding $\phi$ of a graph $G$ into a surface $S$ are the connected components of $S\setminus\phi[G]$.
\end{defn}

\section*{Minimal Separating Sets as Graph Embeddings}

No 0-complex can separate a surface. Moreover, if $M$ separates a surface $S$, then so does the 1-skeleton of $M$ and it follows that every minimal separating set of a surface is a 1-complex. As such, a minimal separating set of a surface can be thought of as the image of a graph embedding into the surface.

\begin{defn}
\label{def:minsepembedding}
If $\phi$ is an embedding of a graph $G$ into a surface $S$ and the image of $\phi$ is a minimal separating set of $S$ then we say $\phi$ is a \emph{minimal separating embedding}.
\end{defn}

Note that if $\phi$ is an embedding of a graph $G$ into $S$, and $\psi$ is an embedding of a graph $H$ into $S$,
and if $\phi[G]$ and $\psi[H]$ are homeomorphic (in the topological sense) then the graphs $G$ and $H$ are graph-homeomorphic.

\begin{defn}
\label{def:graph-underlying-M}
Suppose $M$ is a minimal separating set of a surface
$S$, and let
\begin{align*}
\mathcal{C}_{M}=\{G:\exists \text{ a } \text{ graph embedding } \phi
\textrm{ of } G \textrm{ into } S \textrm{ such that } \phi[G]=M\}\;.
\end{align*}
Then $\mathcal{C}_{M}$ contains a unique homeomorphically irreducible
graph $G_{M}$. We say $G_{M}$ is \emph{the graph underlying}
$M$.
\end{defn}

\begin{defn}
\label{def:min-sep-graph}
We refer to a graph $G$ as a\emph{ minimal
separating graph of $S$} if $G$ is the graph underlying some minimal
separating set of $S$.
\end{defn}

\begin{defn}
\label{def:G_g}
Let $\mathbb{G}_{g}$ denote the set of all
minimal separating graphs of an orientable surface of genus $g$.
\end{defn}

\begin{defn} We will denote the set of all minimal separating graphs
$\cup_{g\geq 0} \,\mathbb{G}_{g}$ by $\mathbb{G}$.
\label{defn:separating graphs}
\end{defn}

We will classify the minimal separating sets of the orientable surfaces
of genus $g\le 3$ in various ways.

\begin{defn} The least separated genus, $\gamma_-(G)$, of a graph $G$ equals the minimal genus of
the oriented surfaces in which $G$ can be embedded in such a way that its embedding
is a minimal separating set. It is $\infty$ if that set of (oriented) surfaces is empty.
\label{defn:mingenus}
\end{defn}

We note here that theorem \ref{thm:separating graphs} shows that $\gamma_-(G)=\infty$
if and only if at least one vertex of $G$ has odd degree or degree zero.

\section*{The Finitude of $\mathbb{G}_{g}$}

By Euler's theorem we know that if $\phi$ is an embedding of $G$ into $S$ then $|V|-|E|+|R|\ge2-2g$, where $V$ is the vertex set of $G$,
$E$ is the edge set of $G$, $R$ is the set of regions of the embedding,
and $g$ is the genus of $S$.

\begin{lem}
\label{lem:L1}
If $\phi$ is a minimal separating embedding of $G$ into $S$
then $|R|=2$.
\end{lem}

\begin{proof}
Since $\phi[G]$ separates $S$, it is clear that $|R|\ge2$.
Now suppose $|R|>2$. Since $S$ is connected, there exist distinct
regions $f_{1}$ and $f_{2}$ such that $\bar{f}_{1}\cap\bar{f_{2}}\ne\emptyset$.
Since $f_{1}$ and $f_{2}$ are distinct regions, $f_{1}\cap f_{2}=\emptyset$.
Thus $bd(\bar{f}_{1})\cap bd(\bar{f_{2}})\ne\emptyset$. Note that
$bd(\bar{f}_{1})\cap bd(\bar{f_{2}})=\phi[\{e_{1}\}\cup...\cup\{e_{n}\}]$
for some $e_{1},...,e_{n}\in E$. Let $e$ be one of the $e_{1},...,e_{n}$
and let $x\in int(\phi(e))$. Then $x\in\phi[G]$, but $S\setminus\left(\phi[G]\setminus\{x\}\right)$
has exactly one less connected component than $S\setminus\phi[G]$.
Since $S\setminus\phi[G]$ has strictly more than 2 connected components,
$S\setminus\left(\phi[G]\setminus\{x\}\right)$ has at least 2 connected
components. Thus $\phi[G]\setminus\{x\}$ separates $S$ and $\phi[G]$
is not a minimal separating set.
\end{proof}

\begin{lem}
\label{lem:L2}
If $\phi$ is a minimal separating embedding of $G$ into $S$
and $e\in E$ then $\phi(e)$ is in the boundary of both regions of $\phi$.
\end{lem}

\begin{proof}
Suppose $\phi$ is a minimal separating embedding of a graph $G$ into a surface $S$.
It is clear that for every edge $e$, $\phi(e)$ must be in the boundary
of at least one region. Suppose, by way of contradiction, that for
some $e\in E$, $\phi(e)$ were in the boundary of exactly one of
the two regions of $\phi$. Then removal of an interior point $x$
of $\phi(e)$ would not decrease the number of connected components.
That is, $S\setminus\left(\phi[G]\setminus\{x\}\right)$ has at least
as many connected components as $S\setminus\phi[G]$. Hence $\phi([G]\setminus\{x\})$
is a separating set of $S$, and it follows that $\phi[G]$ is not
\emph{minimally} separating.
\end{proof}

\begin{thm}
\label{thm:T1}
A graph embedding is a minimal separating embedding
if and only if the following conditions are satisfied:
\begin{enumerate}
\item The graph has no isolated vertices.
\item The embedding has exactly two regions.
\item The image of each edge is in the boundary of both regions of the embedding.
\end{enumerate}
\end{thm}

\begin{proof}
Suppose $\phi$ is a minimal separating embedding of $G$ into $S$.
Condition 1 is immediate. Lemmas \ref{lem:L1} and \ref{lem:L2} demonstrate conditions 2
and 3, respectively. So the forward direction is proved.

Now suppose $\phi$ is an embedding of $G$ into $S$ that satisfies the
three conditions. Since condition 2 is satisfied, we know $\phi[G]$
separates $S$. To show that $\phi[G]$ is a \emph{minimal} separating
set, let $x\in\phi[G]$. Condition 1 implies that $x\in\phi[e]$ for
some $e\in V$. By condition 3 we know that $x$ is in the boundary
of both regions of the embedding. It follows that $S\setminus(\phi[G]\setminus\{x\})$
is connected.
\end{proof}

\begin{lem}
\label{lem:L3}
Suppose $\phi$ is a minimal separating embedding of $G$ into $S$.
Then $|E|-|V|\le2g$ where $g$ is the genus of $S$.
\end{lem}

\begin{proof}
By Euler's theorem and lemma \ref{lem:L1}, $|V|-|E|+2\ge2-2g$.
Thus $|E|-|V|\le2g$ as desired.
\end{proof}

\begin{defn}
\label{def:isolated-loop}
A graph homeomorphic to the graph consisting of one vertex with degree 2 and one edge
(the cycle on one vertex) is called an \emph{isolated loop}.\end{defn}

\begin{lem}
\label{lem:L4}
 For all $g$, and for all $G\in\mathbb{G}_{g}$, if
$v$ is a vertex of $G$ and $deg(v)=2$, then $v$ is an isolated loop.
\end{lem}

\begin{proof}
Suppose $g\ge0$, and that $G\in\mathbb{G}_{g}$. Let
$v$ be a vertex of $G$ such that $deg(v)=2$, and suppose that $v$
is not an isolated loop. Then $v$ is incident to precisely two edges,
and so $v$ can be smoothed. Thus $G$ is not homeomorphically irreducible.
Hence $G\notin\mathbb{G}_{g}$.
\end{proof}

\begin{lem}
\label{lem:L5}
If $G\in\mathbb{G}$, then by Lemma all vertices of $G$ have even degree.
\end{lem}

\begin{proof}
Let $v$ be a vertex and consider a neighborhood around
$\phi(v)$. The edge-ends emanating from $v$ divide the neighborhood
into $deg(v)$ subdivisions. By Lemma \ref{lem:L1}, there are exactly 2 regions
of $\phi$ and each of these subdivisions must be a subset of exactly
one of the two regions. By Lemma \ref{lem:L2}, each of the edge ends must be
in the boundary of both regions. If $deg(v)$ were odd then there
would exist an edge end emanating from $v$ that was in the boundary
of only one of the two regions. Hence $deg(v)$ is not odd.
\end{proof}

In what follows we denote the set of isolated loops of a graph $G$ by $B$.
If $G$ has no isolated loops then $B=\emptyset$. The graph $G$ minus the isolated loops
is denoted by $G\setminus B$, and its set of vertices and set of edges are denoted by
$V(G\setminus B)$ and $E(G\setminus B)$, respectively. The number of components
of $B$ will be denoted by $|B|$.

\begin{lem}
\label{lem:L6}
For all $g\ge0$, and for all $G\in\mathbb{G}_{g}$, $2|V(G\setminus B)|\le|E(G\setminus B)|$.
\end{lem}

\begin{proof}
Let $g\ge0$, and let $G\in\mathbb{G}_{g}$. Let $B$
be the set of isolated loops of $G$. By Lemmas \ref{lem:L4} and \ref{lem:L5}, and the observation
that minimal separating graphs do not contain vertices of degree zero, we have that
$deg(v)\ge4$ for all $v\in V(G\setminus B)$. Thus we have
$4|V(G\setminus B)|+2|B|\le\left[\sum_{v\in V\setminus B}deg(v)\right]+2|B|=2|E(G\setminus B)|+2|B|$.
Thus $2|V(G\setminus B)|\le|E(G\setminus B)|$.
\end{proof}

\begin{lem}
\label{lem:L6.5}
Let $B$ be the set of isolated loops of a graph $G \in\mathbb{G}_{g}$.
Either the graph $G\setminus B$ is empty or it is a minimal
separating graph of a surface of genus $g-|B|$, where $|B|$ is the number of components of $B$.
\end{lem}

\begin{proof}
Suppose $G$ is a minimal separating graph for a surface
$S$ of genus $g$, and let $\phi$ be a minimal separating
embedding of $G$ into $S$. It will suffice to show that if $L$ is any isolated loop
of $G$, and $G\setminus L$ is not empty, then there exists an embedding
of $G\setminus L$ into a surface of genus $g-1$.

Let $L$ be any isolated loop of $G$. Either $\phi[L]$ is simple (ie contractible to a point),
or it is not. If $\phi[L]$ is contractible to a point, then $L$ separates $S$,
and it follows that $G\setminus L$ is empty.

Suppose $G\setminus L$ is not empty. Then $\phi[L]$ is not simple.
So $\phi[L]$ must circumnavigate a handle
of $S$. We will construct a minimal separating embedding of the graph
$G\setminus L$ into a surface of genus $g-1$ using surgery.

Note that $\phi$ satisfies the three conditions of theorem \ref{thm:T1}.
A closed neighborhood of $\phi[L]$ is a cylinder $C$. Remove $C$ from $S$.
Since the Euler characteristic of a cylinder is zero, we have $\chi_S=\chi_{S\backslash C}$.
Note that by condition
3 of theorem \ref{thm:T1}, this has the result of creating exactly two boundary
components (one in each region of $S\phi[G]$. Glue one disk into each
of the two components. Gluing in a disk adds 1 to the characteristic of the surface.
What remains is an embedding of the graph $G\setminus L$
into a surface of surface of genus $g-1$. Note that this embedding
is a minimal separating embedding because it satisfies the three requirements
of theorem 1 -- for it inherited these properties from $\phi$.
\end{proof}

Colloquially, the above lemma states that, in minimal separating graphs,
the only role played by isolated loops is the role of handle-cutter.
The sole exception is the contractible loop. And in that case the
isolated loop is the minimal separating graph itself.

\begin{lem}
\label{lem:L6.75}
For all $g$, and for all $G\in\mathbb{G}_{g}$,
if $B$ is the set of isolated loops of $G$ then $|B|\le g+1$.
\end{lem}

\begin{proof}
Suppose $G\in\mathbb{G}_{g}$, and $\tilde B$ is a set of $g$
isolated loops of $G$. Then by Lemma \ref{lem:L6.5}, the graph
$G\setminus \tilde B$ has a minimal separating embedding in a surface
of genus 0 (a sphere).
Thus $G\setminus \tilde B$ must be an isolated loop.
\end{proof}

\begin{lem}
\label{lem:L6.85}
Let $G\in\mathbb{G}_{g}$,$\phi$ and suppose that $G$ contains a collection of
$g+1$ non-intersecting cycles.
Then $G$ is the graph consisting of $g+1$ isolated loops.
\end{lem}

\begin{proof}
As in the proof of lemma \ref{lem:L6.5}, remove a closed neighborhood $C$ of $g$
of the non-intersecting cycles from the surface $S$ and glue to the resulting boundary
components of $S\backslash C$. The resulting surface is a sphere. Since $G$ contained
another non-intersecting cycle, the sphere also does, separating the sphere.
It follows that the original $g+1$ cycles in $G$ separate $S$. By minimality
$G$ must be equal to that collection of cycles.
\end{proof}

\begin{lem}
\label{lem:L7}
For all $g\ge1$, and for all $G\in\mathbb{G}_{g}$,
if $B$ is the set of isolated loops of $G$ and $G\setminus B$ is not empty,
then $0\le|V(G\setminus B)|\le2g-2|B|$.
\end{lem}

\begin{proof}
Suppose $g\ge1$, and let $G\in\mathbb{G}_{g}$. Let
$B$ be the set of isolated loops of $G$.

If $G\setminus B$ is not empty then, by lemma \ref{lem:L6.5}, the graph $G\setminus B$
is in $\mathbb{G}_{\hat{g}}$, where $\hat{g}=g-|B|$.
It follows by lemma \ref{lem:L3} that $|E(G\setminus B)|-|V(G\setminus B)|\le2\hat{g}$.
By Lemma \ref{lem:L6} we have $2|V(G\setminus B)|\le|E(G\setminus B)|$. Thus we find $|V(G\setminus B)|\le|E(G\setminus B)|-|V(G\setminus B)|\le2\hat{g}=2(g-|B|)$
as desired.
\end{proof}

\begin{thm}
\label{thm:T2}
For all $g$, $\mathbb{G}_{g}$ is finite.
\end{thm}

\begin{proof}
Fix $g$, and let $G\in\mathbb{G}_{g}$. Let $B$ be
the set of isolated loops of $G$. By lemmas \ref{lem:L7} and \ref{lem:L6.75},
$0\le|V(G)|=|V(G\setminus B)|+|B|\le2g-|B|$.
This puts an upper bound on the size of $|V(G)|$ for any graph $G$
in $\mathbb{G}_{g}$. This, in combination with the statement of Lemma
3, provides an upper bound on $|E|$. In particular, $0<|E(G)|\le2g+|V(G)|$.
Thus, $\mathbb{G}_{g}$ is finite.
\end{proof}

\begin{defn}
\label{def:c_g}
For $g\ge1$ define $\mathbb{C}_{g}$ (the `candidate graphs') to be the
set of all graphs $G$ satisfying the following properties:
\begin{enumerate}
\item If $G\setminus B$ is not empty, then $0\le|V(G)|\le2g-|B|$,
\item $0<|E(G)|\le2g+|V(G)|$
\item All vertices of $G$ have even degree greater than zero
\item The only vertices of $G$ with degree two are isolated
\item $G$ has at most $g$ non-intersecting cycles, unless it equals $g+1$ isolated loops.
\end{enumerate}
We refer to $\mathbb{C}_{g}$ as the \emph{set of candidate graphs}
for genus $g$.
\end{defn}

It is clear that for all $g\ge1$, $\mathbb{C}_{g}$ is finite and
$\mathbb{G}_{g}\subset\mathbb{C}_{g}.$ For low genus surfaces, $\mathbb{C}_{g}$
can be generated quite easily using a contemporary computer. To decide
which members of $\mathbb{C}_{g}$ are also in $\mathbb{G}_{g}$,
we provide a method to test whether a given graph in $\mathbb{C}_{g}$
has a minimal separating embedding in a surface of genus $g$.

\section*{Rotation Systems}

\begin{defn}
\label{def:rotation_at_v}
A graph embedding induces, for each vertex $v$, a cyclic ordering of edges incident to $v$.
This ordering is referred to as the \emph{rotation at $v$}, and is unique up to a cyclic
permutation of the edges. Without loss of generality, we will always consider cyclic
orderings in the {\emph clockwise} direction.
\end{defn}

\begin{defn}
\label{def:rotation_system}
A \emph{rotation system} for a graph $G$ consists of an unordered set $\mathcal{R}$
of the vertices $\{v_i\}$ of $G$ together with a cyclically ordered list of all the
edges $e_j$ that are incident to $v_i$. Note that each edge must occur exactly
twice in a rotation system.
\end{defn}

\noindent
{\bf Remark:} This corresponds to the notion of \emph{pure rotation system} as given in \cite{Gross}.
Note that one graph may have many embeddings, leading to non-equivalent rotation systems.
For instance the graphs in Figure 1 have rotation systems given in the table at the
end of this article as graph \# 1, 2, 3, and 8.

\begin{defn}
\label{def:cellular-embedding}
A graph embedding whose regions are homeomorphic to open disks is called a \emph{cellular} embedding.
\end{defn}

Cellular embeddings with the same rotation system are considered equivalent.
As there are only finitely many possible rotations at each vertex, and only
finitely many vertices in a graph, we know there are a finite number of possible
rotation systems for a given graph. Thus the infinite set of all possible cellular embeddings
of a graph into a surface can be partitioned into equivalence classes of finitely many
rotation systems. The following is taken from \cite{Gross}.

\begin{thm}
\label{thm:unique-embedding}
Every rotation system on a graph $G$ defines a unique (up to orientation preserving homeomorphisms)
locally oriented cellular graph embedding $G\rightarrow S$. Conversely every locally
oriented cellular embedding defines a rotation system.
\end{thm}

For the calculations that follow, it will be useful to consider a closed neighborhood
around the image of a graph embedding.

\begin{defn}
\label{def:reduced_band_decomposition}
Let $\phi$ be an embedding of a graph $G$ into a surface $S$. We surround the image of
each vertex with a closed disk, referred to as a \emph{$0$-band}, and surround the image
of each edge with a thin band, referred to as a \emph{$1$-band}. The union of $0$-bands
and $1$-bands forms a closed neighborhood in $S$ that preserves the shape of the graph
G. We refer to the union of the $0$-bands and $1$-bands (omitting the rest of the
surface) as the \emph{reduced band decomposition} of $\phi$.
\end{defn}

For a more thorough discussion of rotation systems and reduced band decompositions see \cite{Gross}.

\begin{defn}
\label{def:boundary_walk}
Every region of an embedding $\phi$ of a graph $G$ into a surface $S$ is bounded by the image of a closed walk (ie a topological circle). This closed walk is unique (up to a cyclic permutation) and the ordered list of edges in the walk is referred to as the \emph{boundary walk} of the region.
\end{defn}

\begin{defn}
\label{def:edges-of-boundary-walk}
Let $\mathcal{B}$ be the reduced band decomposition corresponding
to a graph embedding $\phi$ of a graph $G$ into a surface $S$, and let $b$ be one of the boundary
components of $\mathcal{B}$. Then \emph{$\left\langle b\right\rangle $} denotes the
set of edges of $G$ that appear in the boundary walk of $b$.
\end{defn}

\begin{defn}
\label{def:two-sided-rotaion-system}
Let $\mathcal{R}$ be a rotation system for an embedding
of a graph $G$ with no isolated vertices. Let $b_{1}$, $b_{2}$,
..., $b_{n}$ be the boundary components of the reduced band decomposition
corresponding to $\mathcal{R}$. We refer to $\mathcal{R}$ as a \emph{two-sided rotation
system} if there exists a partition of the set $\{b_{1},b_{2},...,b_{n}\}$
into two sets of closed boundary walks, $B_{1}$ and $B_{2}$, such that for $\forall e\in E$,
$e\in\underset{b\in B_{1}}{\bigcup}\left\langle b\right\rangle$ and
$e\in\underset{b\in B_{2}}{\bigcup}\left\langle b\right\rangle$.
The cardinalities $|B_1|$ and $|B_2|$ will be denoted by $n_1$ and $n_2$.
\end{defn}

\noindent
Colloquially speaking, a two-sided rotation system $\mathcal{R}$ is a one for
which:
\begin{enumerate}
\item All vertices of the underlying graph $G$ have even degree greater than zero and
\item There exists an embedding $\phi$ of $G$ with the following properties:
\begin{enumerate}
\item The rotation system corresponding to $\phi$ is $\mathcal{R}$
\item The embedding $\phi$ has exactly two regions, one with boundary $B_1$ and the other
with boundary $B_2$.
\item The image of each edge can be found in the boundary of each region.
\end{enumerate}
\end{enumerate}

\section*{Identifying Minimal Separating Graphs}

Upon fixing a genus $g$ we can create the finite set $\mathbb{C}_{g}$
of candidate graphs. We know that the set $\mathbb{G}_{g}$ of minimal
separating graphs is contained in $\mathbb{C}_{g}$. Moreover, for
each graph $G\in\mathbb{C}_{g}$ we can generate the finite set $\mathbf{R}_{G}$
of all rotation systems on $G$.

Now suppose $R\in \mathbf{R}_{G}$. We want a way to test whether there exists
a minimal separating embedding $\phi$ from $G$ into $S$ such that:
\begin{enumerate}
\item The rotation system corresponding to $\phi$ is $R$ and
\item The surface $S$ has genus $g$.
\end{enumerate}
Let's begin with the following useful observation.

\begin{cor}
\label{cor:C1} $\mathbb{G}_{0}\subset\mathbb{G}_{1}\subset\mathbb{G}_{2}\subset...$
\end{cor}

\begin{proof}
Let $G\in\mathbb{G}_{k}$ for some $k\ge0$. Let $S_{k}$
be a surface of genus $k$ and let $\phi_{k}$ be a minimal
separating embedding of $G$ into $S_{k}$. If we attach a handle to one of the two regions
of $\phi_{k}$, we increase the genus of the surface by one while
the three conditions of theorem \ref{thm:T1} remain satisfied. The result is
a minimal separating embedding $\phi_{k+1}$ of $G$ into a surface $S_{k+1}$ of genus $k+1$.
\end{proof}

In light of theorem \ref{thm:T1}, it follows that minimal separating graph embeddings
are exactly those with two-sided rotation systems. This is made rigorous with the following theorem.

\begin{thm}
\label{thm:T3}
There exists a minimal separating embedding $\phi$ of a graph $G$ with corresponding
rotation system $\mathcal{R}$ if and only if $\mathcal{R}$ is a two-sided rotation system.
\end{thm}

\begin{proof}
Suppose there exists a minimal separating embedding $\phi$ with corresponding
rotation system $\mathcal{R}$. Then by condition 2 of theorem 1, there are
exactly two regions of $\phi$. Let $r_{1}$ and $r_{2}$ be the regions
of $\phi$. Note that every boundary component of $\mathcal{B}$ lies in exactly
one of the two regions of $\phi$. So form the desired partition of
$\{b_{1},b_{2},...,b_{n}\}$ by defining $B_{1}$ to be the set of
boundary components of $\mathcal{B}$ that lie in region $r_{1}$ and defining
$B_{2}$ analogously. Then condition 3 of theorem \ref{thm:T1} tell us that $\forall e\in E,\forall i\in\{1,2\},\left(e\in\underset{b\in B_{i}}{\bigcup}\left\langle b\right\rangle \right)$,
and the forward direction is proved.

Now suppose we have a two-sided rotation system so that the set
$\{b_{1},b_{2},...,b_{n}\}$ can be partitioned
into two boundaries, $B_{1}$ and $B_{2}$, such that for $\forall e\in E,\forall i\in\{1,2\},\left(e\in\underset{b\in B_{i}}{\bigcup}\left\langle b\right\rangle \right)$.
The number of connected components in $B_1$ ($B_2$) will be denoted by $n_1$ ($n_2$).
We now construct a minimal separating (cellular) embedding $\phi$ with corresponding
rotation system $\mathcal{R}$ using surgery.

By theorem \ref{thm:unique-embedding}, $\mathcal{R}$ defines a cellular embedding
$\phi : G\rightarrow M$ in a surface $M$. This embedding defines (by definition
\ref{def:cellular-embedding}) a reduced band decomposition $\mathcal{D}$ with boundary components
(definition \ref{def:two-sided-rotaion-system}) that are (disjoint) topological circles.
These circles come in two sets, those that can be retracted to $B_1$ and which we will call
$\tilde B_1$ and those than can be retracted to $B_2$ (called $\tilde B_2$).

Now consider two spheres, $A_{1}$ and $A_{2}$. Remove $n_1$
open discs from $A_{1}$ and name the result $\overline{A}_{1}$. Construct
$\overline{A}_{2}$ similarly. We identify the boundary components of $\tilde B_i$
to the boundary components of $\overline{A}_{i}$ through a bijection $\psi_i$:
\begin{equation}
\psi_{i}:\partial(\overline{A}_{i})\to \tilde B_{i}
\label{eq:boundary-bijection}
\end{equation}
The result is a surface $S$ that is the connected sum of the reduced band decomposition $\mathcal{D}$
with the $n_i$-punctured spheres $\overline A_i$:
\begin{equation}
S = {\cal D} \# \overline A_1 \# \overline A_2
\label{eq:connected-sum}
\end{equation}

By construction, there are exactly two regions of the
embedding (the regions corresponding to $\overline{A}_{1}$ and $\overline{A}_{2}$),
and every edge is on the boundary of both regions. So by theorem \ref{thm:T1},
the embedding is a minimal separating embedding.
\end{proof}

Given an embedding $\phi$ of a graph $G$, the boundary
components of the reduced band decomposition can be computed from the rotation
system by a simple combinatorial algorithm called face tracing (section 3.2.6 of \cite{Gross}).

\begin{lem}
\label{lem:two-sided}
Given the representation
\begin{equation*}
R=\quad
\begin{array}{ccc}
v_0 & :& a\;x_1\cdots x_p\;a\;x_{p+1}\cdots x_{k}\\
v_1 & : & \cdots\\
\vdots &&
\end{array}
\end{equation*}
of the rotation system $\mathcal{R}$ of a minimally separating embedding of a graph $G$ it follows that $p$ and $k$ are even.
\end{lem}

\begin{proof} By Lemma \ref{lem:L5}, $p$ is even iff $k$ is even. We show that $p$ is even.

In the embedding edge $a$ must separate a boundary component $b_1\in B_1$
from a boundary component $b_2\in B_2$.
Orient edge $a$ and travel with the orientation along $a$. The boundary component
at your right side (with respect to the orientation) must always point towards
the same component. It follows that the \emph{corner} $(x_k,a)$ corresponds
to the same boundary component as the corner $(a,x_{p+1})$. The same is true for the
other two corners involving edge ends of $a$. Since the corners belonging to boundary
component in $B_1$ must alternate with those belonging to $B_2$, we conclude that $p$
is even.
\end{proof}

\begin{thm}
\label{thm:T4}
A graph $G$ is in $\mathbb{G}_{g}$ for some $g$ if and only if there exists
a two-sided rotation system $\mathcal{R}$ on $G$ such that $g\ge\frac{|E|-|V|+n_1+n_2}{2}-1$
(where $n_1$ and $n_2$ are the numbers of boundary components corresponding to regions
1 and 2 of the reduced band decomposition corresponding to $\mathcal{R}$).
\end{thm}

\begin{proof} $G$ is in $\mathbb{G}_{g}$ if and only if there exists
a two-sided rotation system $\mathcal{R}$ by theorem \ref{thm:T3}.
Equivalently the embedding in the surface $S$ with genus $g$ admits a reduced
band decomposition ${\cal D}$ with $n_1$ bordering on the component $A_1$ of
$S\backslash {\cal D}$ and $n_2$ boundary components bordering of the component $A_2$.
Each of these components $A_i$ has itself genus $g_i$.

The genus of $S$ is the genus of the connected sum in Equation \ref{eq:connected-sum}
in the proof of theorem \ref{thm:T3}. Thus
$\chi_{S}=\chi(\overline A_1)+\chi(\overline A_2)+\chi(D)$. Since the characteristic
of an $n$-punctured surface of genus $g_i$ equals $2-2g_i-n_i$, we have
\begin{equation}
\chi_S = 2-2g_1-n_1+2-2g_2-n_2+|V|-|E|
\label{eq:extra-characteristic}
\end{equation}
Solving for the genus $g_S$ of $S$ gives:
\begin{equation}
g_S = \frac{|E|-|V|+n_1+n_2}{2}-1+g_1+g_2
\label{eq:extra-genus}
\end{equation}
Clearly $g$ is minimized when $g_1=g_2=0$.
\end{proof}

\begin{defn} The minimally separating embedding of $G$ is \emph{irreducible} if equality
holds in Theorem \ref{thm:T4} holds (ie if $g_1=g_2=0$ in Equation \ref{eq:extra-genus}.
\label{def:irreducible-embedding}
\end{defn}

\begin{lem} If $G\in\mathbb{G}_{g}$ consists of two disjoint nonempty subgraphs $\underline G$ and
$\overline G$ we have each of the subgraphs as elements of $\bigcup_{g'<g}\,\mathbb{G}_{g'}$.
\label{lem:disjoint graphs}
\end{lem}

\begin{proof} It is sufficient to prove the lemma for $\underline G$.
Suppose $G$ embeds as a minimal separating set in the surface $S$.
Denote the quantities corresponding to the disjoint subgraphs by underbars and overbars, ie
$|\underline E|$ is the number of edges in $\underline G$ and so forth.
Let $\underline C$ be a closed neighborhood of $\phi[\underline G]$. Then $\phi$ restricted
to $\underline C$ satisfies theorem \ref{thm:T1}. Remove $\underline C$ from $S$.
Let $\underline A_1$ and $\underline A_2$ be the $\underline n_1$ and $\underline n_2$
punctured spheres. We denote by $\underline S$ the connected sum
\begin{equation*}
\underline S = \underline C \# \underline A_1 \# \underline A_2
\end{equation*}
The genus of this surface according to equation \ref{eq:extra-genus} equals
\begin{equation*}
g_{\underline S} = \frac{|\underline E|-|\underline V|+\underline n_1+\underline n_2}{2}-1
\end{equation*}
By the same reasoning the genus of the original surface $S$ equals
\begin{eqnarray*}
g_S &=& \frac{|E|-|V|+n_1+n_2}{2}-1+g_1+g_2 \\
 & = &\frac{|\overline E|-|\overline V|+\overline n_1+\overline n_2+
|\underline E|-|\underline V|+\underline n_1+\underline n_2}{2}-1+g_1+g_2
\end{eqnarray*}
Lemma \ref{lem:L4} implies that $|\overline E|-|\overline V|\geq 0$ and lemma \ref{lem:L2} that
$\overline n_1$ and $\overline n_2$ are at least 1. Thus $g_S\geq g_{\underline S} +1$.
\end{proof}

\begin{lem} If $\underline G$ and $\overline G$ have least separated genus
$\underline g$ and $\overline g$, respectively, then their disjoint union has least separated
genus $\underline g+\overline g + 1$
\label{lem:disjoint graphs2}
\end{lem}

\begin{proof} Let $\underline {\cal R}$ be the rotation system for $\underline G$
that minimizes the number of boundary walks $\underline n_1+\underline n_2$.
The least separated genus $\underline g$ satisfies:
\begin{equation*}
\underline g = \frac{|\underline E|-|\underline V|+\underline n_1+\underline n_2}{2}-1
\end{equation*}
The same holds for $\overline G$. The rotation system for the disjoint union of
$\underline G$ and $\overline G$, is the union of their two rotation systems
${\cal R} =\underline {\cal R} \cup \overline {\cal R}$. The boundary walks
of ${\cal R}$ are exactly the union those of in $\underline {\cal R}$ and those
in $\overline {\cal R}$. Therefore the minimum number of boundary walks in  ${\cal R}$
equals $\underline n_1+\underline n_2+\overline n_1+\overline n_2$. And thus the least separated genus $g$ of $G$ satisfies:
\begin{eqnarray*}
g &=& \frac{|\overline E|-|\overline V|+\overline n_1+\overline n_2+
|\underline E|-|\underline V|+\underline n_1+\underline n_2}{2}-1
\end{eqnarray*}
which implies the result.
\end{proof}

\section*{The Results}

Given a genus $g$, we compute three sets associated with it.
\begin{align*}
\mathbb{I}_g &= \{\text{Connected graphs whose least separated genus equals}\hspace{0.5em} g\}\\
\mathbb{L}_g &= \{\text{Graphs whose least separated genus equals}\hspace{0.5em} g \}\\
\mathbb{G}_g &= \{\text{Graphs that have a minimal separating embedding in genus}\hspace{0.5em} g\}
\end{align*}

\begin{prop} We have: $|\mathbb{I}_2|=17$ and $|\mathbb{I}_3|=161$.
\label{prop:numerical-count}
\end{prop}

\begin{proof}
This count was computer generated. The counting program follows the following steps.
\begin{enumerate}
\item
Generate the finite set of connected graphs in $\mathbb{C}_{g}$ of candidate graphs for
genus $g$ (see definition \ref{def:c_g}, it is important to note that condition 5 is
computationally expensive so a weaker version is used which looks for non-intersecting
loops instead of non-intersecting cycles.).
\item
For each connected graph $G$ in $\mathbb{C}_{g}$ generate the finite set of rotation systems
of $G$ satisfying Lemma \ref{lem:two-sided} up to cyclic permutation.
\item
For each rotation system found in step 2, check whether it is two-sided and
satisfies $g = \frac{|E|-|V|+n_1+n_2}{2}-1$ (see theorem \ref{thm:T4}).
\end{enumerate}
The graph $G$ of step 2 is in $\mathbb{I}_{g}$ if and only if any rotation system
tests affirmatively in step 3 (and did not appear in a surface with lower genus).
We keep a running list of the graphs $G$ that pass the test,
as well as the successful rotation systems. In genus 2, we determine all rotation systems
that pass the test (see the table at the end of the article).
In genus 3, we stop if we have 1 successful rotation system for a graph.
\end{proof}

Using the computer-generated count and some combinatorics yields the main result.

\begin{thm}
The cardinalities of $\mathbb{I}_g$, $\mathbb{L}_g$, and $\mathbb{G}_g$ for $g \leq 3$ are
as follows.
\\
\center
\begin{tabular}{|c|l|c|c|}
\hline
Genus & $\mathbb{I}_g$ & $\mathbb{L}_g$ & $\mathbb{G}_g$\\
\hline
0 & 1 & 1 & 1\\
\hline
1 & 3 & 4 & 5\\
\hline
2 & 17 & 21 & 26\\
\hline
3 & 161 & 191 & 217\\
\hline
\end{tabular}
\end{thm}

\noindent
{\bf Remark:} At the date of this writing, none of the integer sequences $\{1,3,17,161\}$,
$\{1,4,21,191\}$, or $\{1,5,26,217\}$ are listed in the Online Encyclopedia of Integer Sequences
\cite{OEIS}.

\begin{proof}
The numbers for genera 0 and 1 are discussed in \cite{Bernhard/Veerman} and \cite{Veerman/Bernhard}. The cardinalities
of $|\mathbb{I}_2|$ and $|\mathbb{I}_3|$ are given in Proposition \ref{prop:numerical-count}.
We show that the counts for $\mathbb{L}_g$  for $\mathbb{G}_g$ can be expressed in terms of
those for $\mathbb{I}_g$.

To obtain $|\mathbb{L}_g|$ from $|\mathbb{I}_i|_{i\leq g}$, we must add to $|\mathbb{I}_g|$
graphs $G$ with more than one component and such that $\gamma_-(G)=g$.
According to Lemma \ref{lem:disjoint graphs}, each component of such a graph $G$ separates a
lower genus. Lemma \ref{lem:disjoint graphs2} then implies that:
\begin{equation*}
\begin{array}{ccl}
|\mathbb{L}_2| & = & |\mathbb{I}_0||\mathbb{I}_0||\mathbb{I}_0| + |\mathbb{I}_0||\mathbb{I}_1|+|\mathbb{I}_2|\\
|\mathbb{L}_3| & = & |\mathbb{I}_0||\mathbb{I}_0||\mathbb{I}_0||\mathbb{I}_0| + |\mathbb{I}_0||\mathbb{I}_0||\mathbb{I}_1| + |\mathbb{I}_0||\mathbb{I}_2| +
|\mathbb{I}_1||\mathbb{I}_1| + |\mathbb{I}_3|
\end{array}
\end{equation*}

Any graph that minimally separates genus $g-1$, also minimally separates genus $g$
(just add a handle to one of the components of the complement of the graph in a
minimal separating embedding).
Thus the total count of all minimal separating graphs in genus $g$ can be
computed as follows
\begin{equation*}
|\mathbb{G}_g| = \sum_{i=0}^g|\mathbb{L}_i|
\end{equation*}
\end{proof}

\begin{figure}[pbth]
\center
\includegraphics[width=4in]{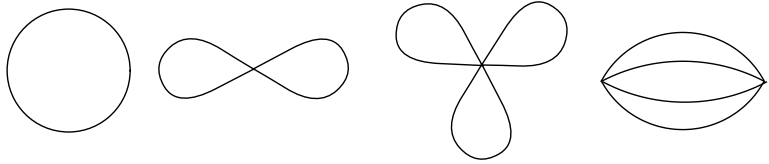}
\caption{ \emph{$\mathbb{I}_0$ consists of one graph (the circle).
The other three constitute the set $\mathbb{I}_1$.
From left to right, respectively graphs 1, 2, 3, and 8 of the table.}}
\label{fig:meta1}
\end{figure}

\begin{figure}[pbth]
\center
\includegraphics[width=2.6in]{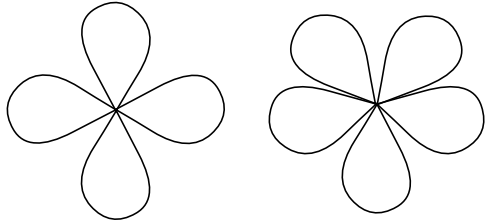}
\caption{ \emph{Graphs in $\mathbb{I}_2$ with one vertex.
From left to right, respectively graphs 4, and 5 of the table.}}
\label{fig:meta2}
\end{figure}

\begin{figure}[pbth]
\center
\includegraphics[width=4.7in]{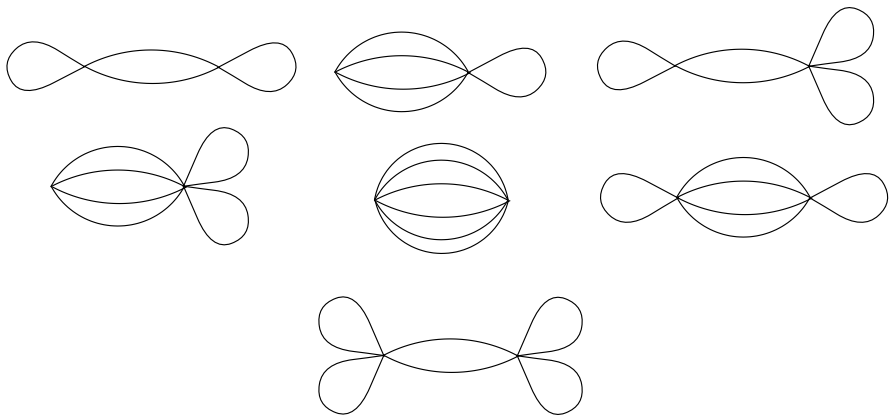}
\caption{ \emph{Graphs in $\mathbb{I}_2$ with two vertices.
From left to right first row, respectively graphs 10, 11, and 12 of the table.
On the second row, graphs 13, 14 , and 15. Last row, graph 16.}}
\label{fig:meta3}
\end{figure}

\begin{figure}[pbth]
\center
\includegraphics[width=4.0in]{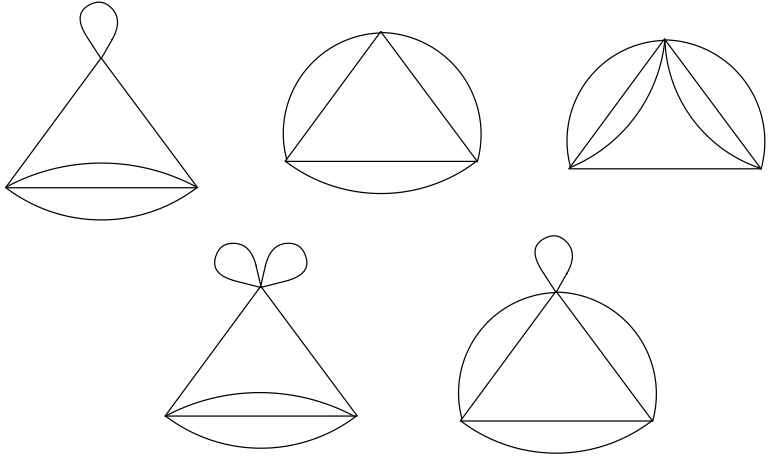}
\caption{ \emph{Graphs in $\mathbb{I}_2$ with three vertices.
From left to right, top row, graphs 19, 20, and 21; below graphs 22 and 23.}}
\label{fig:meta4}
\end{figure}

\begin{figure}[pbth]
\center
\includegraphics[width=4in]{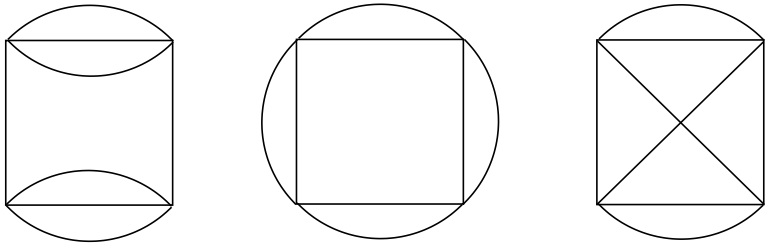}
\caption{ \emph{Graphs in $\mathbb{I}_2$ with four vertices.
From left to right, respectively graphs 24, 25, and 26.}}
\label{fig:meta5}
\end{figure}

The graphs in $\mathbb{I}_0$, $\mathbb{I}_1$, and $\mathbb{I}_2$ are
illustrated in Figures \ref{fig:meta1} through \ref{fig:meta5}.

In Table 1, we list all topologically distinct embeddings of connected minimal separating sets
in genus 2.
As an example, consider graph 3 (see Figure \ref{fig:meta1}), the bouquet of three circles.
This table tells us that there are exactly three minimal separating embeddings of that graph in
a genus 2 torus. If we remove a closed neighborhood $C$
of the embedded graph from the surface $S$, then we obtain two regions that have $n_1$
and $n_2$, respectively, boundary components. So, in the first of these three embeddings,
the boundary of the first regions consists of one (topological) circle, and the boundary
of the second region consists of three circles. In the second embedding, both regions
have a boundary formed by two circles. In the third each regions has a boundary consisting of one
circle. But one of these regions must have a handle attached to it ($g_2=1$). As a consequence,
the same rotation system defines an embedding in a torus of one genus lower (i.e. of genus 1).

In Table 2, we list all graphs in $\mathbb{I}_3$. For each of these graph we list one
rotation system that defines a minimal separating embedding in genus 3.
Note that by definition $(g_1,g_2)=(0,0)$, and this is therefore not listed in the table.
As in Table 1, the numbers of boundary components in $S\backslash C$ for the rotation system
is also given. In contrast to Table 1, however, we do not list all rotation systems that define
a minimal separating embedding for the underlying, as this proved too expensive with our
current methods. Table 2 is split up in partial tables 2.1 through 2.6. The graphs in $\mathbb{I_3}$
with $i$ vertices are grouped in partial Table 2.i. For instance, the complete graph
on 5 vertices, or $K_5$, corresponds to graph 11 in table 2.5. It is a little harder to see
that the octahedron (the platonic solid) corresponds to graph 1 of Table 2.6.

\section*{Final Remarks}

The algorithm in Proposition \ref{prop:numerical-count} is computationally extremely expensive.
The number of (multi)graphs that satisfy Definition \ref{def:c_g} in genus 3 is very large.
In addition, for each graph $G$ there are in principle
${\prod_{v \in V(G)} (\operatorname{deg}(v)!)}$ possible rotation systems.
Even using the techniques mentioned in step 2 of that Proposition to reduce this number,
it still grows far too quickly for our current algorithm to be used for genus $g>3$ on a
contemporary desktop computer.

The graphs for $g\leq2$ are listed
listed and numbered in the Appendix along with their two-sided rotation systems. The set $\mathbb{G}_{1}$ consists of graphs 1, 2, 3, 6, and 8
(in agreement with \cite{Bernhard/Veerman}). The set $\mathbb{G}_{2}$ consists of graphs 1 through 26.

We end with some interesting observations that were however not used in the
computation of the cardinalities discussed in the last section.
The first is that a graph minimally separates
in some genus if and only if all its vertices have even degree.

\begin{thm} A graph $G$ is in $\mathbb{G}$ if and only if all its vertices have
positive, even degree.
\label{thm:separating graphs}
\end{thm}

\begin{proof} One direction is proved by part 1 of theorem \ref{thm:T1} and lemma \ref{lem:L5}.
So we need to prove that if $G$ has all degrees even and positive,
then it minimally separates. By lemma \ref{lem:disjoint graphs2},
it is sufficient to prove this for connected graphs.

Let $G$ be a connected non-empty graph of even degree. Then $G$ admits an Eulerian circuit
$\gamma$ consisting of all its $n$ edges. Label these edges $e_1$, $e_2$... $e_n$, in the
order in which they occur in the circuit $\gamma$. Denote the first vertex of $e_i$, as we
traverse $\gamma$, by $v_i$. Note that $v_i$ maps $\{1,\cdots n\}$ onto the vertex set,
but it is not generally a bijection.

Construct a rotation system ${\cal R}$ for $G$ as follows. At $v_1$ add the edge pair $e_{n}-e_1$.
For all $k\in\{2,\cdots n\}$, at $v_k$ add the edge pair $e_{k-1}-e_k$.
No matter what the ordering of the pairs is at any one vertex in the resulting rotation
system, one boundary walk is given by the edge sequence $\gamma$, which contains
every edge exactly once. Thus the set of all boundary walks can be partitioned
into $B_1=\gamma$ on the one hand, and $B_2$ is the remainder of the boundary walks.
Since all edges must occur twice, $B_2$ also contains every edge exactly once.
And so ${\cal R}$ is two-sided.
\end{proof}

The next observation takes us back to the mediatrices that originally motivated our
study of minimal separating sets (see \cite{Bernhard/Veerman} and \cite{Veerman/Bernhard}).

\begin{thm} For any two-sided rotation system $\mathcal{R}$ corresponding to a graph $G$, there
is a a minimally separating embedding $\phi:G\rightarrow S$ to a (oriented) surface $S$ so that $\phi(G)$
is a mediatrix in that surface.
\label{thm:mediatrices}
\end{thm}

\begin{proof} The proof follows the second part of the proof of Theorem \ref{thm:T3} precisely.
We only need to specify two points $P_1$ and $P_2$ in $S$ and a distance, so that $\phi(G)$ is a mediatrix
in $S$ with respect to those two points. We summarize that argument using the notation of that proof.

First, the surface $M$ defined by the cellular embedding $\phi$ can be embedded in $\mathbb{R}^3$,
and inherit
its distance. Then without loss of generality, for a small enough positive $\epsilon$, the band
decomposition $\cal D$ can be taken to be the closure of the $\epsilon$ neighborhood of $\phi(G)$.
Thus given a point $x\in \phi(G)$, its distance to $\tilde B_1$ equals $\epsilon$ and the same holds
for its distance to $\tilde B_2$. Now for $i$ 1 and 2, choose a point $P_i$ in $\overline A_i$.
The $\overline A_i$ are spheres with $n_i$ disks cut out of them. Thus we can again
give these the inherited distance from $R^3$. For $\Delta$ large enough we can arrange it,
possibly by sewing ``trunks" onto the boundaries of the disks, that the distance from $P_i$
to any point of $\partial(\overline{A}_{i})$ equals $\Delta$. In the connected sum, therefore,
the distance from  $x\in \phi(G)$ to $P_i$ equals $\Delta +\epsilon$.
\end{proof}

For our final remark we need the following Definition (compare with Definition
\ref{defn:mingenus}). Recall the definition of irreducible embedding (Definition \ref{def:irreducible-embedding}).

\begin{defn} The largest irreducibly separated genus, $\gamma_+(G)$, of a graph $G$ equals
the maximal genus of the oriented surfaces in which $G$ can be irreducibly embedded
as a minimally separating set. It is $\infty$ if that set of (oriented) surfaces is either empty
or if their genus has no upper bound.
\label{defn:maxgenus}
\end{defn}

Denote by $\mathbb{G}_E$ the graphs all of whose vertices have even degree. Thus if
$G\in \mathbb{G}_E$, then $\gamma_-(G)$ and $\gamma_+(G)$ are finite by the previous
Theorem. So for these graphs there are always minimally separating embeddings.
There is an interesting relation between these quantities on the one hand and the
\emph{minimum genus} $\gamma_{min}(G)$ and the \emph{maximum genus} $\gamma_{max}(G)$
(see \cite{Chen}) on the other. While the problem of deciding the minimum genus of a graph $G$
is NP-complete (\cite{Thomassen}), there are explicit estimates for the maximum genus (\cite{Xuong}).
Thus the following Theorem enables us to give an explicit lower bound for $\gamma_-(G)$.

\begin{thm} For graphs $G$ in $\mathbb{G}_E$ we have the following relations:
\begin{equation*}
\begin{array}{ccl}
\gamma_-(G)&\geq & |E|-|V|-\gamma_{max}(G)\\
\gamma_+(G)&\leq & |E|-|V|-\gamma_{min}(G)
\end{array}
\end{equation*}
\label{thm:relations}
\end{thm}

\begin{proof} Fix a graph $G\in \mathbb{G}_E$. Let $\mathcal{R}_G$ be the set of rotation
system for that graph, and $\mathcal{R}_G^*\subseteq \mathcal{R}_G$ the set of two-sided
rotation systems for that graph.
Every two-sided rotation system $R$ defines an irreducible minimally separating embedding
in a surface $S$ whose genus $g_S$ satisfies (Equations \ref{eq:extra-characteristic} and
\ref{eq:extra-genus} with $g_1=g_2=0$):
\begin{equation*}
g_S=\dfrac{|E|-|V|+N}{2}-1
\end{equation*}

On the other hand every rotation system $R$ for $G$ defines a cellular embedding
in a surface $C$ by gluing a disk onto each of the $N$ boundary components of $R$.
The genus $g_C$ of $C$ satisfies:
\begin{equation*}
|V|-|E|+N=2-2g_C\quad \Rightarrow \quad |E|-|V|-g_C=\dfrac{|E|-|V|+N}{2}-1
\end{equation*}
The last two equations imply:
\begin{equation*}
\gamma_-(G)= \min_{R\in \mathcal{R}_G^*}\, \dfrac{|E|-|V|+N}{2}-1 \geq
\min_{R\in \mathcal{R}_G }\, \dfrac{|E|-|V|+N}{2}-1 = |E|-|V|-\gamma_{max}(G)
\end{equation*}
The second relation is similarly proved by taking the maximum.
\end{proof}


\clearpage

\begin{table}[ht]
\centering
\label{my-label1}
\begin{tabular}{|c|l|c|c|}
\hline
\multicolumn{4}{|c|}{\large{\bf Table 1}}\\
\hline
\# of Graph	& Rotation System & $(n_1,n_2)$ & $(g_1,g_2)$\\ \hline
1 	& $\begin{array}{l}v_0 : 0, 0\end{array}$ 			&(1,1)&	(0,2), (1,1)		 \\ \hline
2 	& $\begin{array}{l}v_0 : 0, 0, 1, 1 \end{array}$	 	&(1,2)&		(0,1), (1,0)	 \\ \hline
3 	& $\begin{array}{l}v_0 : 0, 0, 1, 1, 2, 2	\end{array}$ 	&(1,3)&	 (0,0)		 \\ \hline
  	& $\begin{array}{l}v_0 : 0, 1, 1, 0, 2, 2	\end{array}$ 	&(2,2)&	 (0,0)		 \\ \hline
 	& $\begin{array}{l}v_0 : 0, 1, 2, 0, 1, 2	\end{array}$ 	&(1,1)&	(0,1)	 \\ \hline
4 	& $\begin{array}{l}v_0 : 0, 0, 1, 2, 3, 1, 2, 3	\end{array}$ 	&(1,2)&	(0,0)	 \\ \hline
  	& $\begin{array}{l}v_0 : 0, 1, 2, 0, 3, 2, 1, 3	\end{array}$ 	&(1,2)&	(0,0)	 \\ \hline
5 	& $\begin{array}{l}v_0 : 0, 1, 2, 0, 1, 3, 4, 2, 3, 4	\end{array}$ &(1,1)& (0,0)   \\ \hline
 	& $\begin{array}{l}v_0 : 0, 1, 2, 0, 3, 2, 4, 3, 1, 4	\end{array}$ &(1,1)& (0,0)   \\ \hline
    & $\begin{array}{l}v_0 : 0, 1, 2, 3, 4, 0, 1, 2, 3, 4	\end{array}$ &(1,1)& (0,0)   \\ \hline
6 	& $\begin{array}{l}v_0 : 0, 0 \\ v_1 : 1, 1	\\ \end{array}$ 	&(2,2)&	(0,1)	 \\ \hline
7 	& $\begin{array}{l}v_0 : 0, 0, 1, 1 \\	v_1 : 2, 2 \\ \end{array}$ 	 &(2,3)& (0,0)      \\ \hline
8 	& $\begin{array}{l}v_0 : 0, 1, 2, 3 \\ v_1 : 0, 1, 2, 3 \\ \end{array}$ 		&(1,1)&	 (0,1) \\ \hline
	& $\begin{array}{l}v_0 : 0, 1, 2, 3 \\ v_1 : 0, 3, 2, 1 \\ \end{array}$ 		&(2,2)&	 (0,0) \\ \hline
9 	& $\begin{array}{l}v_0 : 0, 0 \\ v_1 : 1, 2, 3, 1, 2, 3 \\ \end{array}$ &(2,2)& (0,0)	   \\ \hline
10 	& $\begin{array}{l}v_0 : 0, 0, 1, 2 \\ v_1 : 1, 3, 3, 2 \\ \end{array}$ 	&(2,2)&	(0,0)     \\ \hline
  	& $\begin{array}{l}v_0 : 0, 0, 1, 2 \\ v_1 : 1, 2, 3, 3 \\ \end{array}$ 	&(2,2)&	(0,0)     \\ \hline
11  & $\begin{array}{l}v_0 : 0, 1, 2, 3 \\ v_1 : 0, 4, 2, 1, 4, 3 \\ \end{array}$ &(1,2)&	 (0,0)   \\ \hline
    & $\begin{array}{l}v_0 : 0, 1, 2, 3 \\ v_1 : 0, 4, 4, 1, 2, 3\\ \end{array}$ 	&(1,2)&	 (0,0)   \\ \hline
12 	& $\begin{array}{l}v_0 : 0, 0, 1, 2 \\ v_1 : 1, 3, 4, 2, 3, 4 \\ \end{array}$ 	&(1,2)&	  (0,0) \\ \hline
13 	& $\begin{array}{l}v_0 : 0, 1, 2, 3 \\ v_1 : 0, 4, 5, 1, 2, 3, 4, 5\\ \end{array}$ &(1,1)&(0,0) \\ \hline
    & $\begin{array}{l}v_0 : 0, 1, 2, 3 \\ v_1 : 0, 4, 5, 3, 4, 1, 2, 5\\ \end{array}$ &(1,1)&(0,0) \\ \hline
14 	& $\begin{array}{l}v_0 : 0, 1, 2, 3, 4, 5\\ v_1 : 0, 1, 2, 3, 4, 5\\ \end{array}$ &(1,1)& (0,0)	 \\ \hline
    & $\begin{array}{l}v_0 : 0, 1, 2, 3, 4, 5\\ v_1 : 0, 1, 4, 5, 2, 3\\ \end{array}$ &(1,1)& (0,0)	 \\ \hline
15 	& $\begin{array}{l}v_0 : 0, 1, 2, 0, 3, 4 \\ v_1 : 1, 5, 4, 2, 5, 3 \\ \end{array}$ &(1,1)& (0,0)	 \\ \hline
16 	& $\begin{array}{l}v_0 : 0, 1, 2, 0, 1, 3\\ v_1 : 2, 4, 5, 3, 4, 5 \\ \end{array}$ &(1,1)& (0,0)	 \\ \hline
\end{tabular}
\end{table}
\pagebreak

\begin{table}[ht]
\centering
\label{my-label2}
\begin{tabular}{|c|l|c|c|}
\hline
\multicolumn{4}{|c|}{\large{\bf Table 1 continued}}\\
\hline
\# of Graph	& Rotation System & $(n_1,n_2)$ & $(g_1,g_2)$\\ \hline
17 	& $\begin{array}{l}v_0 : 0, 0 \\ v_1: 1, 1 \\ v_2 : 2, 2 \\ \end{array}$ 	  &(3,3)&(0,0) \\ \hline
18 	& $\begin{array}{l}v_0 : 0, 0 \\ v_1: 1, 2, 3, 4 \\ v_2 : 1, 2, 3, 4 \\ \end{array}$&(2,2)&  (0,0) \\ \hline
19 	& $\begin{array}{l}v_0 : 0, 0, 1, 2 \\ v_1: 1, 3, 4, 5\\ v_2 : 2, 3, 4, 5\\ \end{array}$ &(1,2)&	 (0,0)\\ \hline
20 	& $\begin{array}{l}v_0 : 0, 1, 2, 3 \\ v_1: 0, 4, 5, 1\\ v_2 : 2, 3, 4, 5\\ \end{array}$ &(1,2)&	 (0,0)\\ \hline
    & $\begin{array}{l}v_0 : 0, 2, 1, 3 \\ v_1: 0, 4, 1, 5\\ v_2 : 2, 4, 3, 5\\ \end{array}$ &(1,2)&	 (0,0)\\ \hline
21 	& $\begin{array}{l}v_0 : 0, 1, 2, 3\\ v_1: 0, 4, 5, 6\\ v_2 : 1, 2, 3, 4, 5, 6\\ \end{array}$ &(1,1)&(0,0)\\ \hline
    & $\begin{array}{l}v_0 : 0, 1, 2, 3\\ v_1: 0, 4, 5, 6\\ v_2 : 1, 2, 5, 6, 3, 4\\ \end{array}$ &(1,1)&(0,0)\\ \hline
22 	& $\begin{array}{l}v_0 : 0, 1, 2, 3\\ v_1: 1, 2, 3, 6\\ v_2 : 0, 4, 5, 6, 4, 5\\ \end{array}$ &(1,1)&(0,0)\\ \hline
23 	& $\begin{array}{l}v_0 : 0, 1, 2, 3\\ v_1: 0, 1, 4, 5\\ v_2 : 2, 6, 5, 3, 6, 4\\ \end{array}$&(1,1)&(0,0)\\ \hline
    & $\begin{array}{l}v_0 : 0, 2, 1, 3\\ v_1: 0, 4, 1, 5\\ v_2 : 2, 6, 3, 5, 6, 4\\ \end{array}$&(1,1)&(0,0)\\ \hline
24& $\begin{array}{l}v_0 : 0, 1, 2, 3\\v_1: 4, 5, 6, 7\\v_2 : 0, 4, 5, 6\\v_3 : 1, 2, 3, 7\\ \end{array}$&(1,1)&(0,0)\\ \hline
25& $\begin{array}{l}v_0 : 0, 2, 1, 3\\v_1: 4, 6, 5, 7\\v_2 : 0, 4, 1, 5\\v_3 : 2, 7, 3, 6\\ \end{array}$ &(1,1)&(0,0)\\ \hline
26& $\begin{array}{l}v_0 : 0, 1, 2, 3\\v_1: 0, 4, 5, 6\\v_2 : 1, 4, 5, 7\\v_3 : 2, 3, 6, 7\\ \end{array}$ &(1,1)&(0,0)\\ \hline
  & $\begin{array}{l}v_0 : 0, 1, 2, 3\\v_1: 0, 4, 6, 5\\v_2 : 1, 4, 7, 5\\v_3 : 2, 3, 7, 6\\ \end{array}$ &(1,1)&(0,0)\\ \hline
\end{tabular}
\end{table}

\clearpage

\begin{table}[ht]
\centering
\label{my-label3}
\begin{tabular}{|c|l|c|}
\hline
\multicolumn{3}{|c|}{\large{\bf Table 2.1}}\\
\hline
\# of Graph	& Rotation System & $(n_1,n_2)$ \\ \hline
1 	& $\begin{array}{l}v_0 : 0, 0,1,2,3,1,2,4,5,3,4,5\\ \end{array}$ 	  &(1,2)\\ \hline
2 	& $\begin{array}{l} v_0: 0,1,2,0,1,3,4,2,3,5,6,4,5,6\\ \end{array}$ 	  &(1,1)\\ \hline
\end{tabular}\end{table}

\begin{table}[ht]
\centering
\label{my-label4}
\begin{tabular}{|c|l|c|}
\hline
\multicolumn{3}{|c|}{\large{\bf Table 2.2}}\\
\hline
\# of Graph	& Rotation System & $(n_1,n_2)$ \\ \hline
1 	& $\begin{array}{l}v_0 : 0, 1, 2, 3, 4, 5, 6, 7\\ v_1: 0, 1, 2, 3, 4, 5, 6, 7\\ \end{array}$ 	  &(1,1)\\ \hline
2 	& $\begin{array}{l}v_0 : 0, 1, 2, 3, 4, 5, 6, 0, 1, 7 \\ v_1: 2, 3, 4, 5, 6, 7 \\ \end{array}$&(1,1) \\ \hline
3 	& $\begin{array}{l}v_0 : 0, 1 ,2 ,3, 4, 0, 5, 6\\ v_1: 1, 7, 3, 4 ,6, 2, 7, 5\\ \end{array}$ &(1,1)\\ \hline
4 	& $\begin{array}{l} v_0: 0, 1, 2, 3, 4, 5, 6 ,0 ,1, 2, 3, 7\\ v_1 : 4, 5, 6, 7\\ \end{array}$ &(1,1)\\ \hline
5   & $\begin{array}{l}v_0:0, 1, 2, 3, 4, 0, 5, 6, 1, 2 \\ v_1 : 3, 7, 6, 4, 7, 5\\ \end{array}$ &(1,1)\\ \hline
6 	& $\begin{array}{l}v_0 : 0, 1, 2, 3, 4, 0, 1, 5\\ v_1: 2, 6, 7, 3, 4, 5, 6, 7\\ \end{array}$ &(1,1)\\ \hline
7   & $\begin{array}{l}v_0 : 0, 1, 2, 3, 4, 0,  1, 2, 3, 5\\ v_1: 2, 6, 7, 3, 4, 5, 6, 7\\ \end{array}$ &(1,1)\\ \hline
8 	& $\begin{array}{l}v_0 : 0, 0, 1, 2, 3, 4 ,5, 6\\ v_1: 1, 2, 3, 4, 5, 6\\ \end{array}$ &(1,2)\\ \hline
9 	& $\begin{array}{l}v_0 : 0, 1 ,2, 0, 1, 2, 3, 4, 5, 6\\ v_1: 3, 4, 5, 6\\ \end{array}$&(1,2)\\ \hline
10   & $\begin{array}{l}v_0 : 0, 1, 2, 3, 4, 0, 1, 5\\ v_1: 2, 6, 4, 3, 6, 5 \end{array}$&(1,2)\\ \hline
11  & $\begin{array}{l}v_0 : 0, 1, 2, 3, 4, 0, 1, 2, 3, 5\\v_1: 4 ,6, 6, 5 \end{array}$&(1,2)\\ \hline
12& $\begin{array}{l}v_0 : 0, 1 ,2, 0, 1, 2, 3, 4\\v_1: 3, 5, 6, 4, 5, 6\\ \end{array}$ &(1,2)\\ \hline
13& $\begin{array}{l}v_0 : 0, 1, 2, 0, 1, 2, 3, 4\\v_1: 3, 5, 5, 4\\ \end{array}$ &(1,3)\\ \hline
\end{tabular}
\end{table}

\begin{table}[ht]
\centering
\label{my-label5}
\begin{tabular}{|c|l|c||c|l|c|}
\hline
\multicolumn{6}{|c|}{\large{\bf Table 2.3}}\\
\hline
\# of Graph	& Rotation System & $(n_1,n_2)$&\# of Graph	& Rotation System & $(n_1,n_2)$ \\ \hline
1 	& $\begin{array}{l}v_0 : 0, 1, 2, 3\\ v_1: 0, 4, 4, 1\\ v_2:2, 5, 5, 3 \end{array}$ 	  &(1,4)&17 	 & $\begin{array}{l}v_0 : 0, 1 ,2, 3\\ v_1: 0, 4, 5, 1\\ v_2:2, 6, 7, 3, 4, 5, 6, 7\end{array}$ 	  &(1,2)\\ \hline
2 	& $\begin{array}{l}v_0 : 0, 0, 1, 2 \\ v_1: 1, 3, 3, 4 \\ v_2: 2, 5, 5, 4 \end{array}$&(2,3)&18 	 & $\begin{array}{l}v_0 : 0, 1, 2, 3\\ v_1: 0, 4 ,5, 6, 7, 1\\ v_2:2, 3, 4, 5, 6, 7 \end{array}$ 	  &(1,2) \\ \hline
3 	& $\begin{array}{l}v_0 : 0, 1 ,2, 3\\ v_1:4, 4, 5, 6\\ v_2:0, 1 ,2 ,3 ,5, 6\\ \end{array}$ &(2,2)&19 	 & $\begin{array}{l}v_0 : 0, 1, 2, 3\\ v_1: 0, 4 ,5 ,6, 4 ,1\\ v_2:2, 7, 6, 3, 7, 5 \end{array}$ 	  &(1,2)\\ \hline
4 	& $\begin{array}{l} v_0: 0, 1, 2, 3\\ v_1 :0, 4, 4, 5\\v_2: 1, 6, 3, 4, 2, 6, 5 \end{array}$ &(1,3)&20 	& $\begin{array}{l}v_0 : 0, 1, 2, 3\\ v_1: 0, 4, 5, 1, 4, 5\\ v_2:2, 6 ,7 ,3, 6, 7 \end{array}$ 	  &(1,2)\\ \hline
5   & $\begin{array}{l}v_0:0, 1, 2 ,3\\ v_1 : 0, 4, 4 ,1\\ v_2: 2, 5 ,6 ,3 ,5 ,6\end{array}$ &(2,2)&21 	 & $\begin{array}{l}v_0 : 0, 2 ,3, 1 ,4 ,5\\ v_1: 0, 6 ,6 ,1\\ v_2:2 7, 5 ,3 ,7 ,4 \end{array}$ 	  &(1,2)\\ \hline
6 	& $\begin{array}{l}v_0 : 0 0 1 2\\ v_1: 3, 3 ,4 ,5\\ v_2:1, 6, 4, 2, 6, 5 \end{array}$ &(1,3)&22 	 & $\begin{array}{l}v_0 : 0, 0, 1, 2\\ v_1: 3, 4, 5, 3, 4, 6\\ v_2:1, 7, 5, 2, 7, 6 \end{array}$ 	  &(1,2)\\ \hline
7   & $\begin{array}{l}v_0 :0, 0, 1, 2\\v_1: 1, 3 ,3 ,4 \\v_2: 2, 5, 6, 4 ,5, 6 \end{array}$ &(1,3)&23 	 & $\begin{array}{l}v_0 : 0, 0, 1, 2 \\ v_1: 1, 3, 4, 5, 3, 6\\ v_2:2, 7, 4, 6, 7, 5 \end{array}$ 	  &(1,2)\\ \hline
8 	& $\begin{array}{l}v_0 : 0, 1, 2, 3\\ v_1: 4, 5, 6, 7\\ v_2: 0, 1 ,2 ,3 ,4, 5, 6, 7\end{array}$ &(1,2)&24 	& $\begin{array}{l}v_0 : 0, 0 ,1, 2\\ v_1: 1, 3, 4 ,5, 3 ,6\\ v_2:2, 6, 7, 5, 6, 7 \end{array}$ 	  &(1,2)\\ \hline
9 	& $\begin{array}{l}v_0 : 0, 1, 2, 3\\ v_1:4, 4, 5, 6\\v_2: 0,7, 2, 3, 5, 1, 7, 6 \end{array}$&(1,2)&25 	& $\begin{array}{l}v_0 : 0, 1, 2, 3\\ v_1:4, 5, 6, 7\\ v_2:0, 8, 2, 3, 4, 1, 6, 7, 8, 5 \end{array}$ 	  &(1,1)\\ \hline
10   & $\begin{array}{l}v_0 : 0, 1, 2, 3\\ v_1: 4, 5, 6, 4, 5, 7\\ v_2:0 ,1, 2 ,3 ,6 ,7\\ \end{array}$&(1,2)&26 	& $\begin{array}{l}v_0 : 0, 1, 2, 3\\ v_1: 4, 5, 6,  4,7,8\\ v_2:0,1,2, 5,7, 3,6,8 \end{array}$ 	  &(1,1)\\ \hline
11  & $\begin{array}{l}v_0 : 0, 1 ,2 ,3\\v_1: 0, 4, 5, 6\\v_2:1, 7, 3, 2, 5, 6, 7, 4 \end{array}$&(1,2)&27 	& $\begin{array}{l}v_0 : 0, 1, 2 ,3\\ v_1:4, 5, 6, 4, 5, 7\\ v_2:0,8,2,3,6,1,8,7 \end{array}$ 	  &(1,1)\\ \hline
12& $\begin{array}{l}v_0 : 0, 1, 2, 3\\v_1: 0, 4, 4, 5\\v_2:1, 6, 7, 2, 3, 5, 6, 7 \end{array}$ &(1,2)&28 	 & $\begin{array}{l}v_0 : 0, 1, 2, 3\\ v_1: 0, 4, 5, 6\\ v_2:1, 7, 8, 2, 3, 4, 5 ,6 ,7 ,8 \end{array}$&(1,1)\\ \hline
13& $\begin{array}{l}v_0 : 0, 1, 2, 3\\v_1: 0, 4, 5, 6, 4, 7\\v_2: 1, 2, 3, 6, 5, 7 \end{array}$ &(1,2)&29 	& $\begin{array}{l}v_0 : 0, 1, 2, 3\\ v_1: 0, 4, 5,6,7,8\\ v_2:1,2,3,4,5,6,7,8 \end{array}$ 	  &(1,1)\\ \hline
14 	& $\begin{array}{l}v_0 : 0, 1, 2, 3\\ v_1: 0, 4, 5, 6, 4, 5\\ v_2:1, 7, 3, 2, 7, 6 \end{array}$ 	  &(1,2)&30 	& $\begin{array}{l}v_0 : 0, 1, 2 ,3\\ v_1: 0, 4, 5,6 ,4, 7\\ v_2:1,8,3,6,8,2,5,7 \end{array}$ 	  &(1,1)\\ \hline
15 	& $\begin{array}{l}v_0 : 0, 1, 2, 3 \\ v_1: 0 ,4 ,5 ,6, 4, 7, 6, 5 \\ v_2: 1, 2, 3, 7 \end{array}$&(1,2)&31 	& $\begin{array}{l}v_0 : 0, 1, 2, 3\\ v_1: 0, 4, 5,6,4,5\\ v_2:1,7,8,2,3,6,7,8 \end{array}$ 	  &(1,1) \\ \hline
16 	& $\begin{array}{l}v_0 : 0, 1, 2, 3 ,4 ,5\\ v_1:0, 6, 6, 7\\ v_2:1, 2, 3, 4, 5, 7\\ \end{array}$ &(1,2)&32 	& $\begin{array}{l}v_0 : 0, 1, 2, 3\\ v_1: 0, 4, 6,7,8, 4, 5\\ v_2:1,2,3,6,7,8 \end{array}$ 	  &(1,1)\\ \hline
\end{tabular}
\end{table}

\begin{table}[ht]
\centering
\label{my-label6}
\begin{tabular}{|c|l|c||c|l|c|}
\hline
\multicolumn{6}{|c|}{\large{\bf Table 2.3 continued}}\\
\hline
\# of Graph	& Rotation System & $(n_1,n_2)$&\# of Graph	& Rotation System & $(n_1,n_2)$ \\ \hline
33 	& $\begin{array}{l}v_0 : 0, 1, 2, 3\\ v_1: 0, 4, 5, 6, 7, 8, 4, 5, 6, 7\\ v_2:1,2,3,6,7,8 \end{array}$ 	  &(1,1)&40 	& $\begin{array}{l}v_0 : 0, 1, 2, 3,4,5\\ v_1: 0,6,7,1,2,8\\ v_2:3,4,5,6,7,8\end{array}$ 	  &(1,1)\\ \hline
34 	& $\begin{array}{l}v_0 : 0, 1, 2,3,4,5 \\ v_1: 0,6,7,8,6,7 \\ v_2:1,2,3,4,5,8\end{array}$&(1,1)&41 	 & $\begin{array}{l}v_0 :0,1,3,2,4,5\\ v_1:0,6,7,1,6,2\\ v_2:3,8,4,7,8,5 \end{array}$ 	  &(1,1) \\ \hline
35 	& $\begin{array}{l}v_0 : 0, 1, 2, 3\\ v_1:0,1,4 ,5\\ v_2:2,6,7,8,5,3,6,7,8,4\\ \end{array}$ &(1,1)&42 	 & $\begin{array}{l}v_0 :0,1,2,0,3,4\\ v_1:1,5,6,7,5,6\\ v_2:2,8,3,7,8,4 \end{array}$ 	  &(1,1)\\ \hline
36 	& $\begin{array}{l} v_0: 0, 1, 2, 3\\ v_1 :0, 4,5,1,6,7\\v_2:2,8,4,5,7,3,8,6 \end{array}$ &(1,1)&43 	 & $\begin{array}{l}v_0 :0,1,2,0,3,4\\ v_1:1,5,6,2,5,7\\ v_2:3,8,6,4,8,7\end{array}$ 	  &(1,1\\ \hline
37   & $\begin{array}{l}v_0:0 ,1, 2, 3\\ v_1 : 0, 4 ,5 ,1,4,6\\ v_2: 2,7,8,3,6,5,7,8\end{array}$ &(1,1)&44 	& $\begin{array}{l}v_0 :0,1,3,0,2,4\\ v_1:1,5,6,2,5,6\\ v_2:3,7,8,4,7,8 \end{array}$ 	  &(1,1)\\ \hline
38 	& $\begin{array}{l}v_0 : 0,1,2,3,4,5\\ v_1:0,6,7,1,6,8\\ v_2:2,3,4,5,8,7\end{array}$ &(1,1)&45 	& $\begin{array}{l}v_0 :0,1,2,0,1,3\\ v_1:2,4,5,6,4,5\\ v_2:3,7,8,6,7,8 \end{array}$ 	  &(1,1)\\ \hline
39   & $\begin{array}{l}v_0 :0,2,3,1,4,5\\v_1:0,6,7,1,6,7\\v_2: 2,8,5,3,8,4\end{array}$ &(1,1)&&&\\ \hline
\end{tabular}
\end{table}

\begin{table}[ht]
\centering
\label{my-label7}
\begin{tabular}{|c|l|c||c|l|c|}
\hline
\multicolumn{6}{|c|}{\large{\bf Table 2.4}}\\
\hline
\# of Graph	& Rotation System & $(n_1,n_2)$&\# of Graph	& Rotation System & $(n_1,n_2)$ \\ \hline
1 	& $\begin{array}{l}v_0 : 0, 1, 2, 3\\ v_1: 4,4,5,6\\ v_2:0,7,7,5\\v_3:1,2,3,6\end{array}$ 	  &(2,2)&13 	& $\begin{array}{l}v_0 : 0, 1,2,3\\ v_1: 4,5,6,7\\ v_2:0,1,4,5\\v_3:2,8,6,3,8,7\end{array}$ 	  &(1,2)\\ \hline
2 	& $\begin{array}{l}v_0 : 0, 1, 2, 3\\ v_1: 4,4,5,6\\ v_2:0,7,5,6\\v_3:1,2,3,7\end{array}$ 	  &(2,2)&14 	& $\begin{array}{l}v_0 : 0, 1,2,3\\ v_1: 4,4,5,6\\ v_2:0,1,7,8\\v_3:2,5,8,3,6,7\end{array}$ 	  &(1,2)\\ \hline
3 	& $\begin{array}{l}v_0 : 0, 1, 2, 3\\ v_1: 4,4,5,6\\ v_2:0,7,5,1\\v_3:2,3,7,6\end{array}$ 	  &(1,3)&15 	& $\begin{array}{l}v_0 : 0, 1,2,3\\ v_1: 4,4,5,6\\ v_2:0,1,7,5\\v_3:2,8,6,3,8,7\end{array}$ 	  &(1,2)\\ \hline
4 	& $\begin{array}{l}v_0 : 0, 1, 2, 3\\ v_1: 0,4,4,5\\ v_2:1,6,6,7\\v_3:2,3,5,7\end{array}$ 	  &(2,2)&16 	& $\begin{array}{l}v_0 : 0, 1,2,3\\ v_1: 4,5,6,4,5,7\\ v_2:0,8,6,1\\v_3:2,3,8,7\end{array}$ 	  &(1,2)\\ \hline
5 	& $\begin{array}{l}v_0 : 0, 1, 2, 3\\ v_1: 4,5,6,7\\ v_2:4,5,6,8\\v_3:0,1,2,3,7,8\end{array}$ 	  &(1,2)&17 	& $\begin{array}{l}v_0 : 0, 1,2,3\\ v_1: 0,4,5,6\\ v_2:1,7,8,4\\v_3:2,3,6,7,8,5\end{array}$ 	  &(1,2)\\ \hline
6 	& $\begin{array}{l}v_0 : 0, 1, 2, 3\\ v_1: 4,5,6,7\\ v_2:0,8,8,4\\v_3:1,2,3,5,6,7\end{array}$ 	  &(1,2)&18 	& $\begin{array}{l}v_0 : 0, 1,2,3\\ v_1: 0,4,5,6\\ v_2:1,7,7,4\\v_3:2,8,6,3,8,7\end{array}$ 	  &(1,2)\\ \hline
7 	& $\begin{array}{l}v_0 : 0, 1, 2, 3\\ v_1: 4,5,6,7\\ v_2:0,8,4,5\\v_3:1,2,3,6\end{array}$ 	  &(1,2)&19 	& $\begin{array}{l}v_0 : 0, 1,2,3\\ v_1: 0,4,5,6\\ v_2:1,7,5,4\\v_3:2,8,6,3,8,7\end{array}$ 	  &(1,2)\\ \hline
8 	& $\begin{array}{l}v_0 : 0, 1, 2, 3\\ v_1: 4,5,6,7\\ v_2:0,4,5,6\\v_3:1,8,3,2,8,7\end{array}$ 	  &(1,2)&20 	& $\begin{array}{l}v_0 : 0, 1,2,3\\ v_1: 0,4,4,5\\ v_2:1,6,7,8,6,7\\v_3:2,3,5,8\end{array}$ 	  &(1,2)\\ \hline
9 	& $\begin{array}{l}v_0 : 0, 1, 2, 3\\ v_1: 4,4,5,6\\ v_2:0,7,8,5\\v_3:1,2,3,7,8,6\end{array}$ 	  &(1,2)&21 	& $\begin{array}{l}v_0 : 0, 1,2,3\\ v_1: 4,5,6,7\\ v_2:4,5,8,9\\v_3:0,1,2,6,8,3,7,9\end{array}$ 	  &(1,1)\\ \hline
10 	& $\begin{array}{l}v_0 : 0, 1, 2, 3\\ v_1: 4,4,5,6\\ v_2:0,7,8,5,7,8\\v_3:1,2,3,6\end{array}$ 	  &(1,2)&22 	& $\begin{array}{l}v_0 : 0, 1,2,3\\ v_1: 4,5,6,7\\ v_2:4,5,6,8\\v_3:0,9,2,3,7,1,9,8\end{array}$ 	  &(1,1)\\ \hline
11 	& $\begin{array}{l}v_0 : 0, 1, 2, 3\\ v_1: 4,4,5,6\\ v_2:0,7,5,8,7,6\\v_3:1,2,3,8\end{array}$ 	  &(1,2)&23 	& $\begin{array}{l}v_0 : 0, 1,2,3\\ v_1: 4,5,6,7\\ v_2:0,8,9,4\\v_3:1,2,3,5,6,8,9,7\end{array}$ 	  &(1,1)\\ \hline
12 	& $\begin{array}{l}v_0 : 0, 1, 2, 3\\ v_1: 4,5,6,4,5,7\\ v_2:0,8,6,7\\v_3:1,2,3,8\end{array}$ 	  &(1,2)&24 	& $\begin{array}{l}v_0 : 0, 1,2,3\\ v_1: 4,5,6,7\\ v_2:0,8,9,4,8,9\\v_3:1,2,3,5,6,7\end{array}$ 	  &(1,1)\\ \hline
\end{tabular}
\end{table}

\begin{table}[ht]
\centering
\label{my-label8}
\begin{tabular}{|c|l|c||c|l|c|}
\hline
\multicolumn{6}{|c|}{\large{\bf Table 2.4 continued}}\\
\hline
\# of Graph	& Rotation System & $(n_1,n_2)$&\# of Graph	& Rotation System & $(n_1,n_2)$ \\ \hline
25 	& $\begin{array}{l}v_0 : 0, 1, 2, 3\\ v_1: 4,5,6,7\\ v_2:0,8,4,5\\v_3:1,9,3,6,9,2,7,8\end{array}$ 	  &(1,1)&37 	& $\begin{array}{l}v_0 : 0,2,1,3\\ v_1: 4,6,5,7\\ v_2:0,4,1,5\\v_3:2,8,9,6,8,7,3,9\end{array}$ 	  &(1,1)\\ \hline
26 	& $\begin{array}{l}v_0 : 0, 1, 2, 3\\ v_1: 4,5,6,7\\ v_2:0,8,4,9,8,5\\v_3:1,2,3,6,7,9\end{array}$ 	  &(1,1)&38 	& $\begin{array}{l}v_0 : 0, 1,2,3\\ v_1: 4,5,6,7\\ v_2:0,8,4,1,9,5\\v_3:2,3,6,7,8,9\end{array}$ 	  &(1,1)\\ \hline
27 	& $\begin{array}{l}v_0 : 0, 1, 2, 3\\ v_1: 4,5,6,7\\ v_2:0,4,5,6\\v_3:1,8,9,2,3,7,8,9\end{array}$ 	  &(1,1)&39 	& $\begin{array}{l}v_0 : 0, 1,2,3\\ v_1: 4,5,6,7\\ v_2:0,8,5,1,8,4\\v_3:2,9,7,3,9,6\end{array}$ 	  &(1,1)\\ \hline
28 	& $\begin{array}{l}v_0 : 0, 1, 2, 3\\ v_1: 4,5,6,7\\ v_2:0,8,9,4,5,6\\v_3:1,2,3,8,9,7\end{array}$ 	  &(1,1)&40 	& $\begin{array}{l}v_0 : 0, 1,2,3\\ v_1: 4,6,7,5,8,9\\ v_2:0,1,4,5\\v_3:2,6,8,3,7,9\end{array}$ 	  &(1,1)\\ \hline
29 	& $\begin{array}{l}v_0 : 0, 1, 2, 3\\ v_1: 4,5,6,7,8,9\\ v_2:0,4,5,6\\v_3:1,2,3,7,8,9\end{array}$ 	  &(1,1)&41 	& $\begin{array}{l}v_0 : 0, 1,2,3\\ v_1: 4,5,6,4,7,8\\ v_2:0,1,9,5\\v_3:2,8,6,3,7,9\end{array}$ 	  &(1,1)\\ \hline
30 	& $\begin{array}{l}v_0 : 0, 1, 2, 3\\ v_1: 4,5,6,7,8,9\\ v_2:0,4,5,6,7,8\\v_3:1,2,3,9\end{array}$ 	  &(1,1)&42 	& $\begin{array}{l}v_0 : 0, 1,2,3\\ v_1: 4,5,7,4,6,8\\ v_2:0,1,5,6\\v_3:2,9,8,3,9,7\end{array}$ 	  &(1,1)\\ \hline
31 	& $\begin{array}{l}v_0 : 0, 1, 2, 3\\ v_1: 4,5,7,4,6,8\\ v_2:0,9,5,6\\v_3:1,2,3,9,7,8\end{array}$ 	  &(1,1)&43 	& $\begin{array}{l}v_0 : 0, 1,2,3\\ v_1: 4,5,6,4,5,7\\ v_2:0,1,8,9\\v_3:2,6,9,3,7,8\end{array}$ 	  &(1,1)\\ \hline
32 	& $\begin{array}{l}v_0 : 0, 1, 2, 3\\ v_1: 4,5,6,4,7,8\\ v_2:0,9,5,7,9,6\\v_3:1,2,3,8\end{array}$ 	  &(1,1)&44 	& $\begin{array}{l}v_0 : 0, 1,2,3\\ v_1: 4,5,6,4,5,7\\ v_2:0,1,8,6\\v_3:2,9,7,3,9,8\end{array}$ 	  &(1,1)\\ \hline
33 	& $\begin{array}{l}v_0 : 0, 1, 2, 3\\ v_1: 4,5,6,4,7,8\\ v_2:0,5,7,9,6,8\\v_3:1,2,3,9\end{array}$ 	  &(1,1)&45 	& $\begin{array}{l}v_0 : 0, 1,2,3\\ v_1: 0,4,5,6\\ v_2:1,7,8,4\\v_3:2,9,6,3,8,5,9,7\end{array}$ 	  &(1,1)\\ \hline
34 	& $\begin{array}{l}v_0 : 0, 1, 2, 3\\ v_1: 4,5,6,4,5,7\\ v_2:0,8,9,6\\v_3:1,2,3,8,9,7\end{array}$ 	  &(1,1)&46 	& $\begin{array}{l}v_0 : 0, 1,2,3\\ v_1: 0,4,5,6\\ v_2:1,7,8,9,7,4\\v_3:2,5,8,3,6,9\end{array}$ 	  &(1,1)\\ \hline
35 	& $\begin{array}{l}v_0 : 0, 1, 2, 3\\ v_1: 4,5,6,4,5,7\\ v_2:0,8,9,6,8,9\\v_3:1,2,3,7\end{array}$ 	  &(1,1)&47 	& $\begin{array}{l}v_0 : 0, 1,2,3\\ v_1: 0,4,5,6\\ v_2:1,7,8,4,7,8\\v_3:2,9,6,3,9,5\end{array}$ 	  &(1,1)\\ \hline
36 	& $\begin{array}{l}v_0 : 0, 1, 2, 3\\ v_1: 4,5,6,4,5,7\\ v_2:0,8,6,9,8,7\\v_3:1,2,3,9\end{array}$ 	  &(1,1)&48 	& $\begin{array}{l}v_0 : 0, 1,2,3\\ v_1: 0,4,5,6\\ v_2:1,4,5,7\\v_3:2,8,9,3,6,7,8,9\end{array}$ 	  &(1,1)\\ \hline
\end{tabular}
\end{table}

\begin{table}[ht]
\centering
\label{my-label9}
\begin{tabular}{|c|l|c||c|l|c|}
\hline
\multicolumn{6}{|c|}{\large{\bf Table 2.4 continued}}\\
\hline
\# of Graph	& Rotation System & $(n_1,n_2)$&\# of Graph	& Rotation System & $(n_1,n_2)$ \\ \hline
49 	& $\begin{array}{l}v_0 : 0, 1, 2, 3\\ v_1: 0,4,5,6\\ v_2:1,7,8,4,5,9\\v_3:2,3,6,7,8,9\end{array}$ 	  &(1,1)&52 	& $\begin{array}{l}v_0 : 0, 1,2,3\\ v_1: 0,4,5,6,4,7\\ v_2:1,8,5,9,8,6\\v_3:2,3,7,9\end{array}$ 	  &(1,1)\\ \hline
50 	& $\begin{array}{l}v_0 : 0, 1, 2, 3\\ v_1: 0,4,5,6\\ v_2:1,7,5,8,7,4\\v_3:2,9,6,3,9,8\end{array}$ 	  &(1,1)&53 	& $\begin{array}{l}v_0 : 0, 1,2,3\\ v_1: 0,4,5,6,4,5\\ v_2:1,7,8,9,7,8\\v_3:2,3,6,9\end{array}$ 	  &(1,1)\\ \hline
51 	& $\begin{array}{l}v_0 : 0, 1, 2, 3\\ v_1: 0,4,5,6,7,8\\ v_2:1,4,5,6,7,9\\v_3:2,3,8,9\end{array}$ 	  &(1,1)&  &&\\ \hline
\end{tabular}
\end{table}

\begin{table}[ht]
\centering
\label{my-label10}
\begin{tabular}{|c|l|c||c|l|c|}
\hline
\multicolumn{6}{|c|}{\large{\bf Table 2.5}}\\
\hline
\# of Graph	& Rotation System & $(n_1,n_2)$&\# of Graph	& Rotation System & $(n_1,n_2)$ \\ \hline
1 	& $\begin{array}{l}v_0 : 0, 1, 2, 3\\ v_1: 4,5,6,7\\ v_2:4,8,8,9\\v_3:0,9,5,6\\v_4:1,2,3,7\end{array}$ 	  &(1,2)&10 	& $\begin{array}{l}v_0 : 0, 1,2,3\\ v_1: 4,4,5,6\\ v_2:0,7,8,5\\v_3:1,7,9,6\\v_4:2,3,9,8\end{array}$ 	  &(1,2)\\ \hline
2 	& $\begin{array}{l}v_0 : 0, 1, 2, 3\\ v_1: 4,5,6,7\\ v_2:4,8,8,9\\v_3:0,5,6,7\\v_4:1,2,3,9\end{array}$ 	  &(1,2)&11	& $\begin{array}{l}v_0 : 0, 1,2,3\\ v_1: 0,4,5,6\\ v_2:1,7,8,4\\v_3:2,5,9,7\\v_4:3,6,9,8\end{array}$ 	  &(1,2)\\ \hline
3 	& $\begin{array}{l}v_0 : 0, 1, 2, 3\\ v_1: 4,5,6,7\\ v_2:4,8,9,5\\v_3:0,8,9,6\\v_4:1,2,3,7\end{array}$ 	  &(1,2)&12 	& $\begin{array}{l}v_0 : 0,1,2, 3\\ v_1: 4,5,6,7\\ v_2:4,8,9,10\\v_3:0,8,9,10\\v_4:1,2,3,5,6,7\end{array}$ 	  &(1,1)\\ \hline
4 	& $\begin{array}{l}v_0 : 0, 1, 2, 3\\ v_1:4,5,6,7\\ v_2:4,5,6,8\\v_3:0,9,7,8\\v_4:1,2,3,9\end{array}$ 	  &(1,2)&13 	& $\begin{array}{l}v_0 : 0, 1,2,3\\ v_1: 4,5,6,7\\v_2:4,8,9,10\\v_3:0,5,8,9\\v_4:1,2,3,6,7,10\end{array}$ 	  &(1,1)\\ \hline
5 	& $\begin{array}{l}v_0 : 0, 2,1, 3\\ v_1: 4,6,5,7\\ v_2:4,8,8,9\\v_3:0,5,1,9\\v_4:2,6,3,7\end{array}$ 	  &(1,2)&14 	& $\begin{array}{l}v_0 : 0, 1, 2,3\\ v_1: 4,5,6,7\\ v_2:4,8,9,10\\v_3:0,8,5,6,9,10\\v_4:1,2,3,7\end{array}$ 	  &(1,1)\\ \hline
6 	& $\begin{array}{l}v_0 : 0, 2,1, 3\\ v_1: 4,6,5,7\\ v_2:4,8,5,9\\v_3:0,8,1,9\\v_4:2,6,3,7\end{array}$ 	  &(1,2)&15 	& $\begin{array}{l}v_0 : 0, 1, 2,3\\ v_1: 4,5,6,7\\ v_2:4,8,9,10,8,9\\v_3:0,10,5,6\\v_4:1,2,3,7\end{array}$ 	  &(1,1)\\ \hline
7 	& $\begin{array}{l}v_0 : 0, 1, 2, 3\\ v_1: 4,5,6,7\\ v_2:4,5,8,9\\v_3:0,1,6,8\\v_4:2,3,9,7\end{array}$ 	  &(1,2)&16 	& $\begin{array}{l}v_0 : 0, 1, 2,3\\ v_1: 4,5,6,7\\ v_2:4,8,9,10,8,9\\v_3:0,5,6,7\\v_4:1,2,3,10\end{array}$ 	  &(1,1)\\ \hline
8 	& $\begin{array}{l}v_0 : 0, 1, 2, 3\\ v_1: 4,5,6,7\\ v_2:0,8,9,4\\v_3:1,5,6,8\\v_4:2,3,7,9\end{array}$ 	  &(1,2)&17 	& $\begin{array}{l}v_0 : 0, 1,2,3\\ v_1: 4,6,5,7\\ v_2:4,8,5,9\\v_3:0,8,10,9\\v_4:1,2,3,7,10,6\end{array}$ 	  &(1,1)\\ \hline
9 	& $\begin{array}{l}v_0 : 0, 1, 2, 3\\ v_1:4,5,6,7\\ v_2:0,8,8,4\\v_3:1,5,6,9\\v_4:2,3,7,9\end{array}$ 	  &(1,2)&18 	& $\begin{array}{l}v_0 : 0, 1, 2,3\\ v_1: 4,5,6,7\\ v_2:4,5,8,9\\v_3:0,10,6,8\\v_4:1,2,3,10,7,9\end{array}$ 	  &(1,1)\\ \hline
\end{tabular}
\end{table}

\begin{table}[ht]
\centering
\label{my-label11}
\begin{tabular}{|c|l|c||c|l|c|}
\hline
\multicolumn{6}{|c|}{\large{\bf Table 2.5 continued}}\\
\hline
\# of Graph	& Rotation System & $(n_1,n_2)$&\# of Graph	& Rotation System & $(n_1,n_2)$ \\ \hline
19 	& $\begin{array}{l}v_0 : 0, 1, 2, 3\\ v_1: 4,5,6,7\\ v_2:4,5,8,9\\v_3:0,10,8,6,10,9\\v_4:1,2,3,7\end{array}$ 	  &(1,1)&28 	& $\begin{array}{l}v_0 : 0, 2,1, 3\\ v_1: 4,5,6,7\\ v_2:4,8,10,5,8,9\\v_3:0,6,1,9\\v_4:2,10,3,7\end{array}$ 	  &(1,1)\\ \hline
20 	& $\begin{array}{l}v_0 : 0, 1, 2, 3\\ v_1: 4,5,6,7\\ v_2:4,8,9,5,8,10\\v_3:0,10,9,6
\\v_4:1,2,3,7\end{array}$ 	  &(1,1)&29 	& $\begin{array}{l}v_0 : 0, 1, 2,3\\ v_1: 4,5,6,7\\ v_2:0,8,9,4\\v_3:1,5,10,8\\v_4:2,6,10,3,7,9\end{array}$ 	  &(1,1)\\ \hline
21 	& $\begin{array}{l}v_0 : 0, 1, 2, 3\\ v_1: 4,5,6,7\\ v_2:4,5,8,9\\v_3:0,6,8,10,7,9\\v_4:1,2,3,10\end{array}$ 	  &(1,1)&30 	& $\begin{array}{l}v_0 : 0, 1,2, 3\\ v_1: 4,5,6,7\\ v_2:0,8,9,4\\v_3:1,8,9,5\\v_4:2,10,7,3,10,6\end{array}$ 	  &(1,1)\\ \hline
22 	& $\begin{array}{l}v_0 : 0, 1, 2, 3\\ v_1: 4,5,6,7\\ v_2:4,5,6,8\\v_3:0,9,7,10,9,8\\v_4:1,2,10\end{array}$ 	  &(1,1)&31 	& $\begin{array}{l}v_0 : 0, 1,2,3\\ v_1: 4,5,6,7\\ v_2:0,8,9,4\\v_3:1,5,6,10\\v_4:2,3,7,10,8,9\end{array}$ 	  &(1,1)\\ \hline
23 	& $\begin{array}{l}v_0 : 0, 1, 2, 3\\ v_1: 4,5,6,7\\ v_2:4,8,9,10\\v_3:0,1,5,8\\v_4:2,6,9,3,7,10\end{array}$ 	  &(1,1)&32 	& $\begin{array}{l}v_0 : 0, 1,2,3\\ v_1: 4,5,6,7\\ v_2:0,8,9,4\\v_3:1,5,6,8\\v_4:2,10,7,3,10,9\end{array}$ 	  &(1,1)\\ \hline
24 	& $\begin{array}{l}v_0 : 0, 2,1, 3\\ v_1: 4,6,5,7\\ v_2:4,8,9,10,8,9\\v_3:0,5,1,10\\v_4:2,6,3,7\end{array}$ 	  &(1,1)&33 	& $\begin{array}{l}v_0 : 0, 1,2,3\\ v_1: 4,5,6,7\\ v_2:0,8,10,9,8,4\\v_3:1,5,6,9\\v_4:2,3,7,10\end{array}$ 	  &(1,1)\\ \hline
25 	& $\begin{array}{l}v_0 : 0, 1, 2, 3\\ v_1: 4,5,6,7\\ v_2:4,5,8,9\\v_3:0,1,10,8\\v_4:2,7,9,3,6,10\end{array}$ 	  &(1,1)&34 	& $\begin{array}{l}v_0 : 0, 1, 2,3\\ v_1: 4,5,6,7\\ v_2:0,8,9,4,8,9\\v_3:1,5,6,10\\v_4:2,3,7,10\end{array}$ 	  &(1,1)\\ \hline
26 	& $\begin{array}{l}v_0 : 0, 2,1, 3\\ v_1: 4,6,5,7\\ v_2:4,8,5,9\\v_3:0,8,1,9\\v_4:2,10,6,10,7\end{array}$ 	  &(1,1)&35 	& $\begin{array}{l}v_0 : 0, 1, 2,3\\ v_1: 4,5,6,4,5,7\\ v_2:0,8,9,6\\v_3:1,8,10,7\\v_4:2,3,10,9\end{array}$ 	  &(1,1)\\ \hline
27 	& $\begin{array}{l}v_0 : 0, 1, 2, 3\\ v_1: 4,5,6,7\\ v_2:4,5,8,9\\v_3:0,1,6,8\\v_4:2,10,9,3,10,7\end{array}$ 	  &(1,1)&36 	& $\begin{array}{l}v_0 : 0, 1,2,3\\ v_1: 0,4,5,6\\ v_2:1,7,8,4\\v_3:2,5,9,7\\v_4:3,10,9,6,10,8\end{array}$ 	  &(1,1)\\ \hline
\end{tabular}
\end{table}

\begin{table}[ht]
\centering
\label{my-label12}
\begin{tabular}{|c|l|c||c|l|c|}
\hline
\multicolumn{6}{|c|}{\large{\bf Table 2.6}}\\
\hline
\# of Graph	& Rotation System & $(n_1,n_2)$&\# of Graph	& Rotation System & $(n_1,n_2)$ \\ \hline
1 	& $\begin{array}{l}v_0 : 0, 1, 2, 3\\ v_1: 0, 4,5,6\\ v_2:1,7,8,9\\v_3:2,10,4,7\\v_4:3,8,5,11\\v_5:6,9,10,11\end{array}$ 	  &(1,1)&7 	& $\begin{array}{l}v_0 : 0, 1, 3,2\\ v_1: 4,6,5,7\\ v_2:0,8,9,4\\v_3:1,10,2,11\\v_4:3,5,8,9\\v_5:6,11,7,10\end{array}$ 	  &(1,1)\\ \hline
2 	& $\begin{array}{l}v_0 : 0, 1, 2, 3\\ v_1: 0, 4,5,6\\ v_2:1,8,7,9\\v_3:2,10,4,11\\v_4:3,10,5,7\\v_5:6,8,11,9\end{array}$ 	  &(1,1)&8 	& $\begin{array}{l}v_0 : 0, 1, 2, 3\\ v_1: 4,5,6,7\\ v_2:0,8,9,10\\v_3:1,2,11,4\\v_4:3,11,5,6\\v_5:7,8,9,10\end{array}$ 	  &(1,1)\\ \hline
3 	& $\begin{array}{l}v_0 : 0, 1, 2, 3\\ v_1: 0, 4,6,5\\ v_2:1,7,8,9\\v_3:2,10,4,11\\v_4:3,5,7,8\\v_5:6,11,9,10\end{array}$ 	  &(1,1)&9 	& $\begin{array}{l}v_0 : 0, 1, 2, 3\\ v_1: 4,5,6,7\\ v_2:0,8,9,10\\v_3:1,2,4,11\\v_4:3,5,6,7\\v_5:8,9,10,11\end{array}$ 	  &(1,1)\\ \hline
4 	& $\begin{array}{l}v_0 : 0, 1, 2, 3\\ v_1: 0,6, 4,5\\ v_2:1,7,8,9\\v_3:2,10,11,4\\v_4:3,10,11,5\\v_5:6,7,8,9\end{array}$ 	  &(1,1)&10 	& $\begin{array}{l}v_0 : 0, 1, 3, 2\\ v_1: 4,5,6,7\\ v_2:0,8,4,9\\v_3:1,10,2,11\\v_4:3,5,6,7\\v_5:8,11,9,10\end{array}$ 	  &(1,1)\\ \hline
5 	& $\begin{array}{l}v_0 : 0, 1, 2, 3\\ v_1: 0,5,4,6\\ v_2:1,8,7,9\\v_3:2,10,7,11\\v_4:3,10,4,11\\v_5:5,9,6,8\end{array}$ 	  &(1,1)&11 	& $\begin{array}{l}v_0 : 0, 1, 2, 3\\ v_1: 4,5,6,7\\ v_2:8,9,10,11\\v_3:0,1,2,4\\v_4:3,8,9,10\\v_5:5,6,7,11\end{array}$ 	  &(1,1)\\ \hline
6 	& $\begin{array}{l}v_0 : 0, 1, 2, 3\\ v_1: 0, 4,5,6\\ v_2:1,7,8,9\\v_3:2,7,8,10\\v_4:3,4,5,11\\v_5:6,9,10,11\end{array}$ 	  &(1,1)&12 	& $\begin{array}{l}v_0 : 0, 2,1, 3\\ v_1: 4,6,5,7\\ v_2:0,8,1,9\\v_3:2,10,3,11\\v_4:4,10,5,11\\v_5:6,9,7,8\end{array}$ 	  &(1,1)\\ \hline
\end{tabular}
\end{table}
\end{document}